\documentclass[11pt,a4paper]{amsart}
\usepackage{texmac,enumitem}
\usepackage[all]{xy}
\usepackage{amsmath,mathtools}
\xyoption{web}

\makeatletter
\newcommand*{\doublerightarrow}[2]{\mathrel{
  \settowidth{\@tempdima}{$\scriptstyle#1$}
  \settowidth{\@tempdimb}{$\scriptstyle#2$}
  \ifdim\@tempdimb>\@tempdima \@tempdima=\@tempdimb\fi
  \mathop{\vcenter{
    \offinterlineskip\ialign{\hbox to\dimexpr\@tempdima+1em{##}\cr
    \rightarrowfill\cr\noalign{\kern.5ex}
    \rightarrowfill\cr}}}\limits^{\!#1}_{\!#2}}}

\operators{Aut,Div,Pic,Pid,Id,Cl,ab,Norm,cts,Ext,ia,opp,ie,opp,J,M,skew,com,Core}
\calsymbols{c}{O,C,B,A,R,K,M,D,S,F,X,Y,L,I}
\bbsymbols{b}{E,A}
\fraksymbols{f}{F,p,A,M,m}

\def\aa{{\rm a}}
\def\ee{{\rm e}}
\def\ii{{\rm i}}
\def\Zz{{\rm Z}}
\def\cZ{Z}
\def\cQ{Q}
\def\cG{G}
\def\Grp{{\bf Grp}}
\def\RMod{{}_R{\bf{Mod}}}
\def\AMod{{}_A{\bf{Mod}}}
\def\Rng{{\bf Rng}}

\begin{document}
\subjclass[2010]{11R29, 16H20, 16K20, 18E35}
\title[Commensurability of automorphism groups]
{Commensurability of automorphism groups}
\author{Alex Bartel}
\address{Mathematics Institute, Zeeman Building, University of Warwick,
Coventry CV4 7AL, UK}
\thanks{{\it Keywords.} Calculus of fractions, Cohen--Lenstra heuristics,
commensurability, ring isogenies,
semisimple rings.}
\thanks{The first author gratefully acknowledges the financial support of
the Royal Commission for the Exhibition of 1851.}
\author{Hendrik W. Lenstra Jr.}
\address{Mathematisch Instituut, Universiteit Leiden, Postbus 9512, 2300 RA Leiden, The Netherlands}
\date{\today}

\maketitle
\begin{abstract}
We develop a theory of commensurability
of groups, of rings, and of modules. It allows us, in certain cases, to
compare sizes of automorphism groups of modules, even when those are infinite.
This work is motivated by the Cohen--Lenstra heuristics on class groups.
\end{abstract}

\section{Introduction}\noindent
Often, when a mathematical object is drawn in some ``random''
manner, the probability that it is isomorphic to a given object is
inversely proportional to the size of the automorphism
group of the latter. The Cohen--Lenstra
heuristics \cite{CL84,CM90}, which make predictions on the
distribution of class groups of ``random'' algebraic number
fields, are, as we intend to show, a special case of this rule, provided that
one passes to Arakelov class groups. Now, Arakelov class groups
may have infinitely many automorphisms, so a difficulty arises
in comparing the sizes of their automorphism groups. This difficulty is
resolved in the present paper. We will address the number-theoretic
implications in a later one.

Our main result, formulated as Theorem~\ref{thm:mainintro}
below, expresses that, for certain
pairs of modules $L$ and $M$ over certain types of ring,
one can meaningfully define the ratio of the size of the
automorphism group $\Aut M$ of $M$ to the size of $\Aut L$,
even when their orders $\#\Aut M$ and $\#\Aut L$ are infinite.
If $\Aut L$ can be naturally embedded in $\Aut M$ as a
subgroup of finite index, then the ratio mentioned may be
defined to be that index. Our approach consists of giving a canonical
definition of an ``index of automorphism groups'', to be denoted
by $\ia(L,M)$, in a more general situation.

As a concrete example, we consider modules over group rings.
Denote by $\Z$ the ring of integers, by $\Q$ the field of
rational numbers, by $\Q_{>0}$ the multiplicative group of
positive rational numbers, by $R[G]$ the group ring of a
group $G$ over a ring~$R$, and by $(G:H)$ the index of a
subgroup $H$ of a group~$G$. By ``module'' we shall always
mean ``left module''.
\begin{theorem}\label{thm:QGintro}
Let $G$ be a finite group, let $V$ be a finitely generated
$\Q[G]$-module, and put $\cS=\{L:L$ is a finitely generated
$\Z[G]$-module with $\Q\otimes_\Z L\cong V$ as
$\Q[G]$-modules$\}$. Then there exists a unique function
$\ia\colon\cS\times\cS\to\Q_{>0}$ such that
\begin{enumerate}[leftmargin=*, label={\upshape(\alph*)}]
\item if $L$, $L'$, $M$, $M'\in\cS$ and $L\cong L'$,
$M\cong M'$, then $\ia(L,M)=\ia(L',M')$;
\item if $L$, $M$, $N\in\cS$, then $\ia(L,M)\cdot\ia(M,N)=
\ia(L,N)$;
\item if $M\in\cS$, and $L\subset M$ is a submodule of
finite index, then with $H=\{\sigma\in\Aut M:\sigma L=L\}$ and
$\rho\colon H\to\Aut L$ mapping $\sigma\in H$ to $\sigma|L$, one
has
$$\ia(L,M)={\frac{(\Aut M:H)\cdot\#\ker\rho}{(\Aut L:\rho H)}}.$$
\end{enumerate}
\end{theorem}\noindent
To explain part (c), we remark that it is not hard to show that
one has $L\in\cS$, and that the three cardinal numbers $(\Aut M:H)$,
$\#\ker\rho$, $(\Aut L:\rho H)$ are finite
(see Section~\ref{sec:autcomm}). Since
these three numbers may be thought of as the ratio of the sizes of
$\Aut M$ and $H$, of $H$ and $\rho H$, and of $\Aut L$ and $\rho H$,
respectively, one may think of the expression in (c) as the ratio of
the sizes of $\Aut M$ and $\Aut L$. The same argument shows that one
has indeed $\ia(L,M)=(\#\Aut M)/\#\Aut L$ if $\Aut M$ and $\Aut L$
are finite.

As an example, let $G$ be the trivial group, and put $n=\dim_\Q V$.
Then each $L\in\cS$ is isomorphic to the direct sum of $\Z^n$ with
a finite abelian group $L_0$, and $\Aut L$ is isomorphic to a
semidirect product $\Hom(\Z^n,L_0)\rtimes(\Aut L_0\times\GL(n,\Z))$,
where both $\Hom(\Z^n,L_0)$ and $\Aut L_0$ are finite. Writing
$M\in\cS$ similarly, and ``cancelling'' $\GL(n,\Z)$, one is led to
believe that
$$\ia(L,M)={\frac{\#\Hom(\Z^n,M_0)\cdot\#\Aut
M_0}{\#\Hom(\Z^n,L_0)\cdot\#\Aut L_0}}={\frac{(\#M_0)^n\cdot\#\Aut
M_0}{(\#L_0)^n\cdot\#\Aut L_0}}.$$
Making this informal argument rigorous (see Proposition \ref{prop:fingengps}),
one discovers that if a
function as in Theorem~\ref{thm:QGintro} exists, it must be given by the
formula just stated. However, that this formula does define a
function meeting all conditions, in particular (c), is not obvious.
Likewise, for general $G$ the uniqueness statement of Theorem~\ref{thm:QGintro} is
easy by comparison to the existence statement. Our proof of
Theorem~\ref{thm:QGintro} is given in Section~\ref{sec:proofs}.

There is little doubt that one can prove Theorem~\ref{thm:QGintro} using a suitable
theory of covolumes of arithmetic groups. Instead, we will give an
entirely algebraic proof, obtaining the theorem as a special case of
a much more general result, of which the formulation requires some
terminological preparation.

{\it Isogenies.} A {\it group isogeny\/} is a group homomorphism
$f\colon H\to G$ with $\#\ker f<\infty$ and $(G:fH)<\infty$, and its
{\it index\/} $\ii(f)$ is defined to be $(G:fH)/\#\ker f$. For
a ring $R$, an $R${\it-module isogeny\/} is an $R$-module homomorphism that
is an isogeny as a map of additive groups. A {\it ring isogeny\/}
is a ring homomorphism that is an isogeny as a map of additive
groups. The index of an isogeny of one of the latter two types is defined as
the index of the induced group isogeny on the additive groups.

{\it Commensurabilities.} If $X$, $Y$ are objects of a category
$\cC$, then a {\it correspondence\/} from $X$ to $Y$ in $\cC$ is
a triple $c=(W,f,g)$, where $W$ is an object of $\cC$ and
$f\colon W\to X$ and $g\colon W\to Y$ are morphisms in~$\cC$; we
will often write $c\colon X\rightleftharpoons Y$ to indicate a
correspondence. A {\it group commensurability\/} is a
correspondence $c=(W,f,g)$ in the category of groups for which
both $f$ and $g$ are isogenies, and the {\it index\/} $\ii(c)$ of
such an isogeny is defined to be $\ii(g)/\ii(f)$. For a ring $R$,
one defines $R$-module commensurabilities and their indices
analogously, replacing the category of groups by the category of
$R$-modules. Likewise, one defines ring commensurabilities and
their indices.

{\it Endomorphisms and automorphisms.} Let $R$ be a ring, and
let $c=\!(N,f,g)\colon$
$L\rightleftharpoons M$ be a correspondence
of $R$-modules. We define the {\it endomorphism ring\/} $\End c$
of $c$ to be the subring $\{(\lambda,\nu,\mu)\in(\End
L)\times(\End N)\times(\End M):\lambda f=f\nu$, $\mu
g=g\nu\}$ of the product ring $(\End L)\times(\End
N)\times(\End M)$.  There are natural ring homomorphisms
$\End c\to\End L$ and $\End c\to\End M$ sending
$(\lambda,\nu,\mu)$ to $\lambda$ and $\mu$, respectively;
we shall write $\ee(c)\colon\End L\rightleftharpoons\End M$ for
the ring correspondence consisting of $\End c$ and those two ring
homomorphisms. Similarly, writing $E^\times$ for the multiplicative
group of invertible elements of a ring $E$, we define the
{\it automorphism group\/} $\Aut c$ of $c$ to be the group
$(\End c)^\times$, and we write $\aa(c)\colon\Aut
L\rightleftharpoons\Aut M$ for the group correspondence
consisting of $\Aut c$ and the natural maps $\Aut c\to\Aut L$,
$\Aut c\to\Aut M$.

A \emph{domain} is a non-zero commutative ring in which the product
of any two non-zero elements is non-zero. A ring is \emph{semisimple} if all short
exact sequences of modules over the ring split.

We can now formulate the general result that we announced.

\begin{theorem}\label{thm:mainintro}
Let $Z$ be an infinite domain such that for all non-zero $m\in Z$ the ring
$Z/mZ$ is finite, let $Q$ be the field of fractions of $Z$, let
$A$ be a semisimple $Q$-algebra of finite vector space dimension
over $Q$, let $R\subset A$ be a sub-$Z$-algebra with $Q\cdot R=A$,
and let $L$, $M$ be finitely generated $R$-modules. Then:
\begin{enumerate}[leftmargin=*, label={\upshape(\alph*)}]
\item there is an $R$-module commensurability
$L\rightleftharpoons M$ if and only if the $A$-modules
$Q\otimes_ZL$ and $Q\otimes_ZM$ are isomorphic;
\item if $c\colon L\rightleftharpoons M$ is an $R$-module
commensurability, then $\ee(c)\colon\End L\rightleftharpoons\End M$
is a ring commensurability, and $\aa(c)\colon \Aut
L\rightleftharpoons\Aut M$ is a group commensurability;
\item if $c$, $c'\colon L\rightleftharpoons M$ are
$R$-module commensurabilities, then one has
$$\ii(\ee(c))=\ii(\ee(c')),\;\;\;\;\; \ii(\aa(c))=\ii(\aa(c')).$$
\end{enumerate}
\end{theorem}
\noindent
The proof of Theorem \ref{thm:mainintro} is given in Section \ref{sec:proofs}.
The essential statement is part (c).

The theorem shows that one can define $\ia(L,M)=\ii(\aa(c))$, independently
of $c$, if one has $Q\otimes_ZL\cong_AQ\otimes_ZM$ and $c\colon
L\rightleftharpoons M$ is an $R$-module commensurability. One
deduces the existence part of Theorem~\ref{thm:QGintro}
from Theorem~\ref{thm:mainintro} by putting
$Z=\Z$, $Q=\Q$, $A=\Q[G]$, and $R=\Z[G]$. Other cases that may arise in
applications include localisations and completions of $\Z$ in the r\^ole of
$Z$, and quotients of $Z[G]$ in the r\^ole of~$R$.

Isogenies, commensurabilities, and their indices have many formal properties,
and it is to these that Section \ref{sec:2} is devoted. Among other things,
we define a notion of \emph{equivalence} of correspondences and, under certain
conditions, the \emph{composition} $d\circ c$ of two correspondences $d$ and $c$.
The index of a commensurability depends only on its equivalence class, and it is
multiplicative in composition of commensurabilities. We introduce, for each
object $L$ in the category under discussion, a group $\cG_L$ of which the
elements are the equivalence classes of commensurabilities $L\rightleftharpoons L$.
The group $\cG_L$ plays an important r\^ole in the paper. It may
be thought of as the automorphism group of $L$ in a ``category of fractions'' \cite{GZ},
which is obtained by formally inverting all isogenies in our category.
We also recall in Section \ref{sec:2} an explicit construction of that category
of fractions: the morphisms are equivalence classes of \emph{skew
correspondences}, which are correspondences $(W,f,g)$ in which $f$ is an isogeny.

Section \ref{sec:ringisog}, on ring isogenies, culminates in the
following result, which is proved as 
Theorem \ref{thm:O}. We shall use it to pass from endomorphism rings of
module commensurabilities to automorphism groups.
\begin{theorem}\label{thm:unitsintro}
Let $E\to F$ be a ring isogeny. Then the induced group
homomorphism $E^\times\to F^\times$ is a group isogeny. If in addition
the map $E\to F$ is surjective, then so is the induced map
$E^\times\to F^\times$.
\end{theorem}\noindent
In Section \ref{sec:semisimple} we prove a property
of the rings $R$ appearing in Theorem \ref{thm:mainintro} that allows
us to apply the results of Section \ref{sec:2} to the category of finitely
generated $R$-modules.
\begin{theorem}\label{thm:noethintro}
Let $R$ be a ring as in the statement of Theorem~\ref{thm:mainintro}.
Then $R$ is left-noetherian and right-noetherian.
\end{theorem}\noindent
For a proof, see Theorem \ref{thm:Q}.
The point of Theorem~\ref{thm:noethintro} is that $R$ is not
required to be finitely generated as a $Z$-module. As an aside,
we characterise, in Theorem \ref{thm:R}, the rings $Z$ satisfying
the hypotheses of Theorem~\ref{thm:mainintro}.

Section \ref{sec:unitssemisimple} furnishes the \emph{deus ex machina} of
the paper.
\begin{theorem}\label{thm:finexpintro}
Let $B$ be a semisimple ring that is finitely generated as a module
over its centre $\Zz(B)$. Then $B^\times/(\Zz(B)^\times[B^\times,B^\times])$ is
an abelian group of finite exponent.
\end{theorem}\noindent
This is proved as Theorem \ref{thm:X}.
In fact, we prove an explicit version of Theorem \ref{thm:finexpintro}.
A {\it central simple algebra\/} over a field $k$ is a ring $B$
that is simple in the sense that it has precisely two two-sided
ideals; that has centre equal to $k$; and that has finite dimension
as a vector space over~$k$; it is a well-known result \cite[(7.22)]{CR1} that,
under these conditions, that dimension is a square.
\begin{theorem}\label{thm:expintro}
Let $k$ be a field, and let $B$ be a central simple algebra over~$k$.
Let the dimension of $B$ as a vector space over $k$ be $d^2$, where
$d$ is a positive integer. Then the group $B^\times/(k^\times[B^\times,B^\times])$ is
abelian of exponent dividing~$d$.
\end{theorem}\noindent
Our proof of Theorem~\ref{thm:expintro} (see Theorem \ref{thm:W}) makes use
of Wedderburn's factorisation theorem for polynomials over division rings.
Theorem~\ref{thm:finexpintro} is an immediate consequence of Theorem~\ref{thm:expintro}.

In Section \ref{sec:skewcorres} we place ourselves in the situation of
Theorem \ref{thm:mainintro}, but replacing the semisimplicity assumption
on $A$ by the condition that $R$ be left-noetherian; by Theorem \ref{thm:noethintro}
this is a weaker condition. We apply the construction of Section \ref{sec:2}
to the category of finitely generated $R$-modules, and obtain a ``category
of fractions'' with the same objects, but with morphisms given by equivalence
classes of skew correspondences. Elaborating upon a well-known argument that
is ascribed to Serre, we prove that there is an equivalence of the
latter category with the category of finitely generated $A$-modules that
sends an $R$-module $L$ to the $A$-module $\cQ\otimes_{\cZ} L$. This has
two important consequences.
The first is part (a) of Theorem \ref{thm:mainintro}, which is contained in
Theorem \ref{thm:parta}. The second is that, for a finitely generated
$R$-module $L$, the group $\cG_L$ introduced in Section \ref{sec:2}
may be identified with the group $\Aut_A(\cQ\otimes_{\cZ} L)$.

Section \ref{sec:autcomm} uses the same hypotheses on $A$ and $R$ as Section
\ref{sec:skewcorres}. It starts off with the proof that, for any commensurability
$c\colon L\rightleftharpoons M$ of finitely generated $R$-modules, the
correspondence $\ee(c)\colon \End L\rightleftharpoons \End M$ is a ring
commensurability; by Theorem \ref{thm:unitsintro}, one then also obtains a group
commensurability $\aa(c)\colon\Aut L\rightleftharpoons \Aut M$. This
proves part (b) of Theorem \ref{thm:mainintro}. Next, we prove in Theorem
\ref{thm:compos} that,
for commensurabilities $c\colon L\rightleftharpoons M$ and
$d\colon M\rightleftharpoons N$ of finitely generated $R$-modules, one has
\begin{eqnarray*}
\ii(\ee(d\circ c))=\ii(\ee(d))\ii(\ee(c)),\quad\quad
\ii(\aa(d\circ c))=\ii(\aa(d))\ii(\aa(c)).
\end{eqnarray*}
This result at once allows us to reduce the proof of Theorem \ref{thm:mainintro}(c)
to the special case that $L=M$,
and shows that $\ii\circ\ee$ and $\ii\circ\aa$ give rise to group homomorphisms
$\cG_L\rar \Q_{>0}$; the statement of Theorem \ref{thm:mainintro}(c) is equivalent
to these homomorphisms being trivial. If we write $B=\End_A(\cQ\otimes_{\cZ}L)$,
then Section \ref{sec:skewcorres} enables us to identify $\cG_L$ with
$B^\times=\Aut_A(\cQ\otimes_{\cZ}L)$ and to prove that the homomorphisms are trivial
on the subgroup $\Zz(B)^\times$ of $B^\times$.

In Section \ref{sec:proofs}, the assumption that $A$ be semisimple is brought
back in. It implies that the ring $B$ just defined is also semisimple. Since
the group homomorphisms $\ii\circ\ee$ and $\ii\circ\aa$ are not only trivial on 
$\Zz(B)^\times\subset B^\times$, but also on the commutator subgroup $[B^\times,B^\times]$,
Theorem \ref{thm:mainintro}(c) becomes an immediate consequence of
Theorem \ref{thm:finexpintro}. We give an example to show that, unlike parts
(a) and (b), part (c) of Theorem \ref{thm:mainintro} may fail if $R$ is
left-noetherian, but $A$ is not semisimple.
In the same section, we prove Theorem \ref{thm:QGintro} by putting $R=\Z[G]$;
as far as we are aware, this special case of Theorem \ref{thm:mainintro} is
essentially as hard as the general case.

\begin{acknowledgements}
We would like to thank the referees for helpful comments.
\end{acknowledgements}

\section{Isogenies and commensurabilities}\label{sec:2}\noindent
This section is devoted to formal properties of isogenies
and commensurabilities, and of their indices.

We begin by recalling a basic notion from category theory, for which
we refer to \cite[Ch. I, \S 11]{Lang}.
Let $L\stackrel{f}{\rightarrow}M\stackrel{g}{\leftarrow}N$
be a diagram in a category $\cC$. We say that $(L\times_MN,p_0,p_1)$ is
a \emph{fibre product} of $L$ and $N$ over $M$ if $L\stackrel{p_0}{\leftarrow}L\times_MN
\stackrel{p_1}{\rar}N$ is a diagram in $\cC$ with the property that
$fp_0=gp_1$, and with the universal property that for any diagram
$L\stackrel{h}{\leftarrow}X \stackrel{j}{\rar}N$ that satisfies $fh=gj$,
there exists a unique morphism $i\colon X\rar L\times_MN$ such that
$h=p_0i$ and $j=p_1i$. When a fibre product exists, it is unique up to a
unique isomorphism, so in that case we may speak of \emph{the} fibre product
of $L$ and $N$ over $M$.
In the category $\Grp$ of groups,
the fibre product of $L\stackrel{f}{\rightarrow}M\stackrel{g}{\leftarrow}N$
exists, and it is given by
$$
L\times_M N=\{(l,n)\in L\times N:f(l) = g(n)\},
$$
with $p_0$ and $p_1$ being the projection maps to $L$ and $N$, respectively.

Throughout this section
$\cC$ will denote a category in which for every diagram
$L\stackrel{f}{\rightarrow}M\stackrel{g}{\leftarrow}N$ the fibre product
of $L$ and $N$ over $M$ exists, equipped
with a functor $\cC\rar \Grp$ that
preserves fibre products. The main examples we have in mind are the category of
groups with the identity functor,
the category of rings with the functor that sends a ring to its underlying
additive group, and the category of finitely generated left $R$-modules for a
left-noetherian ring $R$,
with the functor that sends an $R$-module to its underlying abelian group. 

An \emph{isogeny} in $\cC$ is a morphism that becomes an isogeny
in $\Grp$. A \emph{commensurability} in $\cC$ is a correspondence
in $\cC$ that becomes a commensurability in $\Grp$.
We will often think of an isogeny $f\colon L\rar M$ as a special case of a
commensurability, which we will denote by $c_f$,
namely $c_f=(L,\id,f)\colon L\rightleftharpoons M$.

The \emph{index}
$\ii(f)$ of an isogeny $f$ in $\cC$ is defined to be the index
of the image of $f$ in $\Grp$, and the index of a commensurability is
defined analogously, as in the introduction.

For each of the results \ref{prop:isog} -- \ref{prop:composition}
below, it will be clear that it holds for $\cC$
if it holds for $\Grp$. We will
therefore tacitly assume that $\cC=\Grp$ in the proofs of those results.

\begin{proposition}\label{prop:isog}
Let $L$, $M$, $N$ be objects in $\cC$ and let
$h$ be the composition of two morphisms $L\stackrel{f}{\rar}M\stackrel{g}{\rar}N$.
If two of $f$, $g$, $h$ are isogenies, then so is the third.
Moreover, we then have $\ii(h)=\ii(g)\ii(f)$.
\end{proposition}
\begin{proof}
We have an exact sequence of pointed sets
\begin{eqnarray*}
1\rar \ker f\rar \ker h \rar \ker g \rar M/fL\rar N/hL \rar N/gM\rar 1,
\end{eqnarray*}
in which each map has the property that all its non-empty fibres have
equal cardinality. Hence, any term that sits between two finite
sets in the above sequence is itself finite. The first assertion
of the proposition easily follows. Moreover, if all terms
in the sequence are finite, then the alternating product of their
cardinalities is 1, which proves the second assertion.
\end{proof}

\begin{definition}\label{def:compos}
Let $c=(X,f,g)\colon L\rightleftharpoons M$
and $d=(Y,h,j)\colon M\rightleftharpoons N$ be correspondences in $\cC$.
We define the \emph{composition} of $c$ with $d$ by
$$
d\circ c=(X\times_M Y,f\circ p_0,j\circ p_1)\colon L\rightleftharpoons N,
$$
where $p_0$, $p_1$ are the canonical morphisms from
$X\times_M Y$ to $X$, respectively $Y$.
\end{definition}
\begin{remark}\label{rmrk:assoc}
It follows from the universal property of fibre products, and a routine diagram
chase, that composition
of correspondences is associative up to canonical isomorphism.
\end{remark}

\begin{proposition}\label{prop:fibre}
Let $X\stackrel{g}{\rar} M\stackrel{h}{\leftarrow} Y$ be morphisms in $\cC$,
and suppose that $h$ is an isogeny. Let $(W=X\times_MY,p_0,p_1)$ be the
fibre product of $X$ and $Y$ over $M$. Then:
\begin{enumerate}[leftmargin=*, label={\upshape(\alph*)}]
\item the morphism $p_0$ is an isogeny;
\item if the image of $g$ in $\Grp$ has finite kernel, then so does the image of $p_1$;
\item if $g$ is an isogeny, then so is $p_1$.
\end{enumerate}
\end{proposition}
\begin{proof}
We first prove part (a). We have
$$
\ker p_0 = \{(1,y)\in X\times Y:h(y) = g(1)=1\} \cong \ker h,
$$
which is finite by assumption.
Further, the kernel
of $g\colon X\rar M/hY$ is equal to $p_0W$, so $(X:p_0 W)\leq (M:hY)$, which is
also finite. So $p_0$ is an isogeny.

Similarly, $\ker p_1\cong \ker g$, which proves part (b).
Finally, part (c) is symmetric in $X$ and $Y$, and so follows from part (a).
\end{proof}

\begin{definition}\label{def:skew}
A \emph{skew correspondence} is a correspondence $c=(X,f,g)$ in which
$f$ is an isogeny.
\end{definition}

\begin{proposition}\label{prop:composition}
If $c\colon L\rightleftharpoons M$ and $d\colon M \rightleftharpoons N$ are
skew correspondences, respectively commensurabilities, then
$d\circ c\colon L \rightleftharpoons N$ is a skew correspondence,
respectively a commensurability. Moreover, if $c$ and $d$ are
commensurabilities, then we have
$$
\ii(d\circ c) = \ii(d)\ii(c).
$$
\end{proposition}
\begin{proof}
The first two assertions follow immediately from Propositions \ref{prop:fibre}
and \ref{prop:isog}. The third one follows from Proposition \ref{prop:isog}
and a routine diagram chase, which we leave to the reader.
\end{proof}\noindent
We will now use skew correspondences in order to construct a category
$\cC_{\skew}$ in which all isogenies are invertible. One can show that
the class $\cI$ of isogenies in our category $\cC$ ``admits a calculus
of right fractions'' in the language of Gabriel and Zisman
\cite[Chapter I, Section 2]{GZ}; our $\cC_{\skew}$ is nothing but their
``category $\cC[\cI^{-1}]$ of fractions''.

\begin{definition}\label{def:equiv}
Let $c=(X,f,g)\colon L\rightleftharpoons M$ and
$d=(Y,h,j)\colon L\rightleftharpoons M$ be two correspondences.
We say that $c$ and $d$ are \emph{equivalent} if there exists
a commensurability $(W,p,q)\colon X\rightleftharpoons Y$ such that
$fp=hq$ and $gp=jq$. We will call such a commensurability
an \emph{equivalence} between $c$ and $d$.
\end{definition}

\begin{proposition}\label{prop:equivrel}
Being equivalent in the sense of Definition \ref{def:equiv} is an equivalence
relation.
\end{proposition}
\begin{proof}
The relation is clearly symmetric. Reflexivity is also clear, since
an equivalence between $(X,f,g)$ and itself is
given by $(X,\id,\id)\colon X\rightleftharpoons X$. Transitivity
follows from Proposition \ref{prop:composition}.
\end{proof}\noindent
Note that
Definition \ref{def:equiv} describes the smallest equivalence relation
on the set of correspondences $L\rightleftharpoons M$ for which
$(X,f,g)$ is equivalent to $(W,fp,gp)$ whenever $p\colon W\rar X$ is an isogeny.

\begin{definition}\label{def:inverse}
The \emph{inverse} of a correspondence $c=(X,f,g)\colon L\rightleftharpoons M$
is defined to be $c^{-1}=(X,g,f)\colon M\rightleftharpoons L$.
\end{definition}

\begin{lemma}\label{lem:equivinverse}
Let $c$, $c'\colon L\rightleftharpoons M$ and $d\colon M \rightleftharpoons N$
be correspondences. Then:
\begin{enumerate}[leftmargin=*, label={\upshape(\alph*)}]
\item the correspondence $(d\circ c)^{-1}\colon N\rightleftharpoons L$
is equivalent to the composition $c^{-1}\circ d^{-1}$;
\item if $c$ is equivalent to $c'$, then $c^{-1}\colon M\rightleftharpoons L$
is equivalent to $(c')^{-1}$.
\end{enumerate}
\end{lemma}
\begin{proof}
The proof is easy, and is left to the reader.
\end{proof}

\begin{proposition}\label{prop:equivcompos}
Let $c$, $c'\colon L\rightleftharpoons M$ and $d$, $d'\colon M\rightleftharpoons N$
be correspondences. Suppose that $c$ is equivalent to $c'$,
and $d$ is equivalent to $d'$. Then $d\circ c$ is equivalent
to $d'\circ c'$.
\end{proposition}
\begin{proof}
Let $c=(X,f,g)$, $d=(Y,h,j)$.

First, we prove the proposition in the special
case that $d'=d$, and $c'=(W,fp,gp)$, where
$p\colon W\rar X$ is an isogeny.
Let $(X\times_M Y,p_0,p_1)$ be the fibre product of the diagram
$X\stackrel{g}{\rar} M\stackrel{h}{\leftarrow} Y$, and let
$(W\times_M Y,p_0',p_1')$ be the fibre
product of the diagram $W\stackrel{gp}{\rar} M\stackrel{h}{\leftarrow} Y$.
Thus $d\circ c=(X\times_MY,fp_0,jp_1)$,
and $d\circ c'=(W\times_MY,fpp_0',jp_1')\colon L\rightleftharpoons N$.
Since $gpp_0' = hp_1'$, the universal property of fibre products
guarantees the existence of a unique map
$i\colon W\times_M Y\rar X\times_M Y$ with the property that 
$pp_0'=p_0i$ and $p_1' = p_1i$:
$$
\xymatrix{
& & W\times_M Y\ar[dl]_{p_0'}\ar[d]^{i}\ar[ddr]^{p_1'}&\\
& W\ar[d]^p & X\times_M Y \ar[dl]^{p_0}\ar[dr]_{p_1}& &\\
& X\ar[dl]^f\ar[dr]_g & & Y\ar[dl]^h\ar[dr]_j &\\
\;\;\;\;L\;\; & & M & & \;\;N.\;\;\;\;
}
$$
Moreover, it is easy to see that $(W\times_MY,p_0',i)$
is the fibre product of the diagram
$W\stackrel{p}{\rar} X\stackrel{p_0}{\leftarrow}X\times_MY$. It follows
from Proposition \ref{prop:fibre} that $i$ is an isogeny, which proves
that $d\circ c$ is equivalent to $d\circ c'$.

Now, we prove the proposition in the special case that $d=d'$,
and $c'$ is arbitrary. Write $\sim$ for the equivalence
relation between correspondences. Let $c'=(X',f',g')$, and let
$(W,p,q)\colon X\rightleftharpoons X'$ be an equivalence between $c$ and $c'$.
Since $p$ is an isogeny, we have $c\sim (W,fp,gp)=(W,f'q,g'q)$, and since $q$ is
an isogeny, we have $(W,f'q,g'q)\sim c'$. 
We deduce from the special case of the proposition that we just proved
that $d\circ c\sim d\circ (W,fp,gp)\sim d\circ c'$.

Now, we prove the general case. By Lemma \ref{lem:equivinverse}
and by the special case we just proved, we have
$$
(d\circ c)^{-1} \sim c^{-1}\circ d^{-1}\sim c^{-1}\circ (d')^{-1}\sim (d'\circ c)^{-1}.
$$
It therefore follows from Lemma \ref{lem:equivinverse}(b),
that $d\circ c\sim d'\circ c$. By the special case of the proposition that we
proved already, we also have $d'\circ c \sim d'\circ c'$, and the proposition follows.
\end{proof}

\begin{proposition}\label{prop:equiv}
If $c$ and $d$ are two equivalent commensurabilities, then $\ii(c)=\ii(d)$.
\end{proposition}
\begin{proof}
This is an immediate consequence of Proposition \ref{prop:isog}.
\end{proof}\noindent
The term ``inverse'' is justified by the following result.
\begin{proposition}\label{prop:inverse}
Given a commensurability $c=(X,f,g)\colon L\rightleftharpoons M$, the composition
$c^{-1}\circ c\colon L\rightleftharpoons L$ is equivalent
to the commensurability $(L,\id,\id)$, and the composition
$c\circ c^{-1}\colon M\rightleftharpoons M$ is equivalent to the commensurability
$(M,\id,\id)$.
\end{proposition}
\begin{proof}
First, we prove the assertion on $c^{-1}\circ c$.
By definition, 
$c^{-1}\circ c=(X\times_MX,fp_0,fp_1)\colon L\rightleftharpoons L$,
where $(X\times_MX,p_0,p_1)$ is the fibre product of the diagram
$X\stackrel{g}{\rar} M \stackrel{g}{\leftarrow}X$.
By the universal property of the fibre product, we have a unique map
$i\colon X\rar X\times_MX$ with the property that $p_0i=p_1i=\id\colon X\rar X$.
$$
\xymatrix@R=4.5ex{
& X\ar[d]_(0.4)i\ar@/^5ex/[dd]^(0.3)f & \\
& X\times_M X\ar[dl]_(.6){p_0}\ar[dr]^(.6){p_1} & \\
X \ar[d]_{f}\ar[dr]^(.3){g} & L\ar[dl]_(.3){\id}\ar[dr]^(.3){\id} & X\ar[d]^f\ar[dl]_(.3){g}\\
L & M & L
}
$$
Since $g$ is an isogeny, it follows from Proposition \ref{prop:fibre} that
$p_0$ is an isogeny. By Proposition \ref{prop:isog}, the morphism $i$ is also an isogeny,
so $(X,i,f)\colon X\times_MX\rightleftharpoons L$ defines an equivalence between
$c^{-1}\circ c$ and $(L,\id,\id)$.

The claim for
$c\circ c^{-1}$ follows by applying the result just proved to $c^{-1}$ in place of $c$.
\end{proof}

\begin{definition}\label{def:skewcom}
We define $\cC_{\skew}$ to
be the category with the same objects as in $\cC$, and where, for objects
$L$, $M$, the morphisms from $L$ to $M$ are the equivalence classes
of skew correspondences $L\rightleftharpoons M$. We also
define $\cC_{\com}$ to be
the category with the same objects, and where the morphisms from $L$ to $M$
are the equivalence classes of commensurabilities $L\rightleftharpoons M$.
It follows from Remark \ref{rmrk:assoc} and Propositions \ref{prop:equivrel}
and \ref{prop:equivcompos}, that these are indeed categories, i.e. that
composition of morphisms
is well-defined and associative.
\end{definition}\noindent
Proposition \ref{prop:inverse} implies that $\cC_{\com}$ is a (generally large)
groupoid, i.e. every morphism in $\cC_{\com}$ is an isomorphism. In fact,
we have the following sharper result.

\begin{proposition}\label{prop:maxgroupoid}
The category $\cC_{\com}$ is the maximal subgroupoid of $\cC_{\skew}$.
\end{proposition}
\begin{proof}
Let $c=(X,f,g)\colon L\rightleftharpoons M$ be a skew correspondence,
and let $d=(Y,h,j)\colon M\rightleftharpoons L$ be a two-sided inverse
in $\cC_{\skew}$. So $d\circ c$ is equivalent to the commensurability
$(L,\id,\id)\colon L\rightleftharpoons L$, while $c\circ d$ is equivalent
to $(M,\id,\id)\colon M\rightleftharpoons M$, and in particular both
compositions are commensurabilities. We wish to prove that $g$ is then necessarily
an isogeny, and for this it is enough to assume that $\cC=\Grp$.

Let $(Y\times_L X,p_0,p_1)$
be the fibre product of the diagram $Y\stackrel{j}{\rar} L\stackrel{f}{\leftarrow}X$,
and let $(X\times_M Y,p_0',p_1')$
be the fibre product of the diagram $X\stackrel{g}{\rar} M\stackrel{h}{\leftarrow}Y$.
Since $c\circ d$ is a commensurability, the morphism $gp_1$ is an isogeny,
so $(M:gX)$ is finite. Also, since $d\circ c$ is a commensurability,
the morphism $jp_1'$ is an isogeny, so $\ker p_1'$ is finite.
But $\ker p_1'=\{(x,1)\in X\times Y: g(x) = h(1) = 1\}\cong \ker g$.
So $g$ is an isogeny, as claimed.
\end{proof}

\begin{theorem}\label{thm:G_L}
Let $L$ be an object in $\cC$. Then, the set $\cG_L$ of equivalence classes
of commensurabilities $L\rightleftharpoons L$ forms a group under composition,
and the map $\ii$ induces a group homomorphism $\cG_L\rar \Q_{>0}$.
\end{theorem}
\begin{proof}
The first assertion follows from the fact that $\cG_L=\Hom_{\cC_{\com}}(L,L)$.
The second assertion follows from Propositions \ref{prop:composition} and
\ref{prop:equiv}.
\end{proof}

\section{Ring isogenies}\label{sec:ringisog}\noindent
In the present section we prove that an isogeny of rings induces
an isogeny of multiplicative groups.

We begin by recalling some standard ring theoretic
facts, which can be found in \cite{Lam}.

\begin{definition}
The \emph{Jacobson radical} of a ring $E$, denoted by $\J(E)$, is the 
intersection of the maximal left ideals of $R$.
\end{definition}

\begin{proposition}
Let $E$ be a ring, and $y\in E$. Then the following are equivalent:
\begin{enumerate}[leftmargin=*, label={\upshape(\alph*)}]
\item $y\in \J(E)$;
\item $y$ is contained in every maximal right ideal of $E$;
\item $y$ annihilates every simple left $E$-module;
\item $y$ annihilates every simple right $E$-module;
\item $1-xyz\in E^\times$ for all $x$, $z\in E$.
\end{enumerate}
\end{proposition}
\begin{proof}
See \cite[\S 4]{Lam}.
\end{proof}

\begin{lemma}\label{lem:J}
Let $I$ be a two-sided ideal of $E$ with $I\subset \J(E)$. Then
the map $E^\times\rightarrow (E/I)^\times$ is surjective. Moreover,
$u\in E$ is a unit if and only if $u+I$ is a unit in $E/I$.
\end{lemma}
\begin{proof}
Let $u+I$ be a unit in
$E/I$, and let $v+I$ be its inverse. Then we have $uv$,
$vu\in 1+I\subset 1+\J(E)\subset E^\times$, so $u$ has both a right
inverse, namely $v(uv)^{-1}$, and a left inverse, namely
$(vu)^{-1}v$. It follows that $u$ is a unit in~$E$.
\end{proof}\noindent
A ring is called \emph{simple} if it has exactly two two-sided ideals.
A ring $E$ is called \emph{semisimple} if all
short exact sequences of left $E$-modules split.

Any semisimple ring is a finite direct product of simple rings. If $E$ is
a semisimple ring, then the opposite ring $E^{\opp}$ is also semisimple.
A left-artinian ring is semisimple if and only if its Jacobson radical is 0.
If $E$ is an arbitrary ring, then $\J(E/\J(E))=0$. In particular, if $E$
is left-artinian, then $E/\J(E)$ is semisimple. All these facts can
be found in \cite[\S 3 and \S 4]{Lam}.

The next lemma is also proved as \cite[Lemma 2.6]{NormalBasis}. We
give an alternative proof.
\begin{lemma}\label{lem:K}
Let $E$ and $F$ be rings, let $E\to F$ be a surjective ring
homomorphism, and suppose that $E$ is left-artinian. Then the
induced group homomorphism $E^\times\to F^\times$ is surjective.
\end{lemma}
\begin{proof}
First, suppose that $E$ is semisimple. Then
$E$ can be written as the product of finitely many simple rings.
Since the kernel of $E\to F$ is a two-sided ideal of $E$, it must
be a subproduct, and $F$ may then be identified with the complementary
subproduct. Surjectivity of $E^\times\to F^\times$ is now obvious.

We pass to the general case. Denote by $I$ the image of $\J(E)$ in
$F$, which is a two-sided ideal of~$F$. The map $E\to F$ induces a
surjective ring homomorphism $E/\J(E)\to F/I$, where the ring
$E/\J(E)$ is semisimple, so by the first part of the proof the
induced map $(E/\J(E))^\times\to(F/I)^\times$ is surjective. By Lemma \ref{lem:J},
the map $E^\times\to(E/\J(E))^\times$ is also surjective, so the map $E^\times\to(F/I)^\times$
induced by the composed ring homomorphism $E\to F/I$ is surjective
as well. Now let $v\in F^\times$, and choose $u\in E^\times$ that maps to
$v+I\in(F/I)^\times$. Then the image $w$ of $u$ in $F$ satisfies
$w\equiv v\bmod I$, so $w^{-1}v$ belongs to the image $1+I$ of
$1+\J(E)$ in $F$. Let $x\in1+\J(E)$ map to $w^{-1}v$. Then $ux$
maps to $ww^{-1}v=v$, and we have $ux\in E^\times$ because $u\in E^\times$
and $x\in1+\J(E)\subset E^\times$. This proves surjectivity of
$E^\times\to F^\times$, as required.
\end{proof}\noindent

\begin{lemma}\label{lem:L}
Let $E$ be a ring, and let $I$, $J\subset E$ be two-sided ideals.
Then the kernel of the natural ring homomorphism 
$E\to(E/I)\times(E/J)$ equals $I\cap J$, and its image is the
fibre product $E/I\times_{E/(I+J)}E/J$.
\end{lemma}\noindent
The proof is straightforward, and is left to the reader.

\begin{lemma}\label{lem:M}
Let $E\to F$ be a surjective ring isogeny. Then the induced
group homomorphism $E^\times\to F^\times$ is surjective.
\end{lemma}
\begin{proof}
Let $I$ be the kernel of the isogeny $E\to F$. Then $I$ is finite,
and we may identify $F$ with $E/I$. Write $\End I$ for the
endomorphism ring of the additive group of $I$. Let $J$ and $R$,
respectively, be the kernel and the image of the ring homomorphism
$E\to\End I$ sending $a\in E$ to the map $x\mapsto ax$. Then $R$,
being a subring of $\End I$, is a finite ring, $J$ is a two-sided
ideal of $E$, and we have a ring isomorphism $E/J\to R$. By Lemma
\ref{lem:L}, the combined map $E\to F\times R$ induces a ring isomorphism
$\varphi\colon E/(I\cap J)\to F\times_{E/(I+J)}R$. Now we first
prove that the map $(E/(I\cap J))^\times\to F^\times$ is surjective. Let
$u\in F^\times$. Write $v$ for the image of $u$ in $(E/(I+J))^\times$. Since
$R$ is finite and hence left-artinian, by Lemma \ref{lem:K} we can
choose $w\in R^\times$ mapping to $v\in(E/(I+J))^\times$. Then $(u,w)$ belongs to
$F^\times\times_{(E/(I+J))^\times}R^\times=(F\times_{E/(I+J)}R)^\times$, so
$\varphi^{-1}(u,w)$ is a unit of $E/(I\cap J)$ that maps to
$u\in F^\times$. This proves that $(E/(I\cap J))^\times\to F^\times$ is surjective.
{}From $(I\cap J){\cdot}(I\cap J)\subset JI=0$ it follows that for
each $x\in I\cap J$ the element $1+x$ has inverse $1-x$ and therefore
belongs to~$E^\times$; this implies $I\cap J\subset\J(E)$, so by
Lemma \ref{lem:J} the map $E^\times\to(E/(I\cap J))^\times$ is surjective. The
composed map $E^\times\to F^\times$ is then surjective as well.
\end{proof}\noindent
Part (a) of the following lemma also appears as \cite[Lemma 1]{Lewin}. We
give a new proof here.
\begin{lemma}\label{lem:N}
Let $E$ be a subring of a ring $F$ such that the index $(F:E)$ of
additive groups is finite. Then:
\begin{enumerate}[leftmargin=*, label={\upshape(\alph*)}]
\item the ring $F$ has a two-sided ideal $I$ with
$I\subset E$ for which the ring $F/I$ is finite;
\item the index $(F^\times:E^\times)$ is finite.
\end{enumerate}
\end{lemma}
\begin{proof}
(a)
Put $I=\{x\in F:FxF\subset E\}$. Then $I$ is a
two-sided ideal of $F$ that is contained in~$E$, and we proceed to
show that $I$ has finite index in~$F$. Put $J=\{x\in F:Fx\subset E$
and $xF\subset E\}$. Then we have $I\subset J\subset E\subset F$.
Denote by $D$ the finite abelian group $F/E$, by $\End D$ its
endomorphism ring, and by $(\End D)^{\opp}$ the ring opposite to
$\End D$. Both of these rings are finite. The natural left and right
$E$-module structures on $D$ induce a ring homomorphism
$E\to(\End D)\times(\End D)^{\opp}$ of which $J$ is the kernel. It
follows that $J$ is a two-sided ideal of $E$ of finite index in~$E$.
There is a well-defined group homomorphism
\beq
J & \to & \Hom(D\otimes_\Z D,D)\\
x & \mapsto & ((y+E)\otimes(z+E)\mapsto yxz+E)
\eeq
for $x\in J$, $y,z\in F$.
Its kernel is
$I$, and since $D\otimes_\Z D$ is finite, the group $J/I$ is
finite. Because each of $F/E$, $E/J$, $J/I$ is finite, the ring
$F/I$ is finite. This proves (a).

(b) Let $I$ be as in (a). Then $F/I$ and $(F/I)^\times$ are finite,
so the kernel $K$ (say) of the natural group homomorphism
$F^\times\to(F/I)^\times$ has finite index in $F^\times$. If $x\in K$, then
$x^{-1}\in K$, so both $x$ and $x^{-1}$ are in $1+I$, which is
contained in~$E$. This proves $K\subset E^\times$, so $E^\times$ has finite
index in $F^\times$ as well. This proves (b).
\end{proof}\noindent
We can now prove Theorem \ref{thm:unitsintro} of the introduction, which reads as
follows.

\begin{theorem}\label{thm:O}
Let $E\to F$ be a ring isogeny. Then the induced group
homomorphism $E^\times\to F^\times$ is a group isogeny. If in addition
the map $E\to F$ is surjective, then so is the induced map
$E^\times\to F^\times$.
\end{theorem}
\begin{proof}
The last assertion is Lemma \ref{lem:M}. For the first
assertion, let $I$ and $D$ be the kernel, respectively the image of the
map $E\to F$. Then the kernel of $E^\times\to D^\times$ is contained in
$1+I$ and therefore finite, and by Lemma \ref{lem:M} the image is all of
$D^\times$. Hence $E^\times\to D^\times$ is a group isogeny. Further, the
inclusion map $D^\times\to F^\times$ is obviously injective, while by
Lemma \ref{lem:N}(b) the index of $D^\times$ in $F^\times$ is finite. Hence
$D^\times\to F^\times$ is a group isogeny. By Proposition \ref{prop:isog},
the composed map $E^\times\to F^\times$ is also a group isogeny.
\end{proof}

\section{Residually finite domains}\label{sec:semisimple}\noindent
This section is devoted to some properties of infinite domains all of
whose proper quotients are finite.
\begin{lemma}\label{lem:P}
Let $Z$ be a domain such that for all non-zero $m\in Z$
the ring $Z/mZ$ is finite, let $Q$ be the field of fractions of $Z$,
let $V$ be a finite dimensional $Q$-vector space, and let $L$
be a sub-$Z$-module of~$V$. Then for all non-zero $m\in Z$ the
$Z$-module $L/mL$ is finite of order dividing $\#(Z/mZ)^{\dim_QV}$,
with equality if $L$ is finitely generated and $Q{\cdot}L=V$.
\end{lemma}
\begin{proof}
Let $m\in Z$ be non-zero. First suppose that
$L$ is finitely generated. Let $S\subset L$ be a finite subset
that generates it as a $Z$-module, let $T\subset S$ be a maximal
subset that is linearly independent over~$Q$, and let $M\subset L$
be the $Z$-module generated by~$T$. Then $M$ is $Z$-free of rank
$\#T$, so $M/mM$ is finite of order $\#(Z/mZ)^{\#T}\leq
\#(Z/mZ)^{\dim_QV}$, with equality
if $T$ is a $Q$-basis of $V$ or, equivalently, if $Q{\cdot}L=V$. By
maximality of $T$, we can, for each $s\in S$, choose a non-zero
element $m_s\in Z$ such that $m_ss\in M$, and
$m'=\prod_{s\in S}m_s$ is then a non-zero element of $Z$ satisfying
$m'L\subset M$. Because $M/m'M$ is finite, its subgroup $m'L/m'M$
is finite as well, and since the latter group is isomorphic to
$L/M$ and to $mL/mM$, we find that $L/M$ and $mL/mM$ are finite of
the same order. The group $L/mM$ is finite of order
$\#(L/M){\cdot}\#(M/mM)$, so $L/mL$ is also finite, of order
$$\frac{\#(L/M){\cdot}\#(M/mM)}{\#(mL/mM)}=\#(M/mM)=\#(Z/mZ)^{\#T}\le\#(Z/mZ)^{\dim_QV},$$
with equality if $Q{\cdot}L=V$.

Passing to the general case, let $U$ be the set of finitely
generated sub-$Z$-modules $L'$ of~$L$, which is a directed partially
ordered set by inclusion. Then $L$ is the injective limit of
all $L'\in U$, and $L/mL$ is the injective limit of the modules
$L'/mL'$, all of which have order dividing $\#(Z/mZ)^{\dim_QV}$.
The injective limit has then also order dividing the same number.
This completes the proof of Lemma \ref{lem:P}.
\end{proof}\noindent
We now prove Theorem \ref{thm:noethintro}.

\begin{theorem}\label{thm:Q}
Let $Z$ be a domain such that for all non-zero $m\in Z$ the ring
$Z/mZ$ is finite, let $Q$ be the field of fractions of $Z$, let
$A$ be a semisimple $Q$-algebra of finite vector space dimension
over $Q$, and let $R\subset A$ be a sub-$Z$-algebra with $Q\cdot
R=A$. Then $R$ is left-noetherian and right-noetherian.
\end{theorem}
\begin{proof}
Let $I$ be a left ideal of~$R$. Then $Q{\cdot}I$ is
a left $A$-ideal, so by semisimplicity of $A$ it is a direct
summand of the left $A$-module~$A$. Thus the endomorphism ring
of the latter module contains an idempotent with image
$Q{\cdot}I$. Since the endomorphisms of the left $A$-module $A$
are the right multiplications by elements of $A$, it is
equivalent to say that we can choose an idempotent $e\in A$ with
$Ae=Q{\cdot}I$. We have $e\in Q{\cdot}I$, so we can choose a
non-zero element $m\in Z$ with $me\in I$. Multiplying the chain
of inclusions $Rme\subset I\subset R$ by $e$ on the right, which
when restricted to $I$ is just the identity map, we obtain
$Rme\subset I\subset Re$, where $Rme=mRe$ because $m$ is
central. By Lemma \ref{lem:P}, the group $Re/mRe$ is finite, so $I/Rme$ is
finite as well. Hence $I$ is, as a left $R$-module, generated by
$me$ together with a finite set, and is therefore finitely
generated. This proves that $R$ is left-noetherian. Applying
this result to $A^{\opp}$ and $R^{\opp}$, we find that $R$ is
right-noetherian as well.
\end{proof}

\begin{example}\label{ex:notNoeth}
If we assume $Z\ne Q$, then the semisimplicity
condition on $A$ is actually necessary for the conclusion of
Theorem \ref{thm:Q} to be valid for all~$R$. To see this, assume that
$A$ is not semisimple, or equivalently that $\J(A)\ne0$, and
choose a sub-$Z$-algebra $T\subset A$ that is finitely
generated as a $Z$-module and satisfies $Q{\cdot}T=A$. Then the
ring $R=T+\J(A)$ is not left-noetherian because $\J(A)$ is
not finitely generated as a left $R$-ideal. If it were, then
the non-zero $Q$-vector space $\J(A)/\J(A)^2$ would be
finitely generated as a $T$-module and hence as a $Z$-module,
which for $Z\ne Q$ is impossible.
\end{example}

\begin{lemma}\label{lem:finquodomain}
Let $Z$ be an infinite commutative ring. Suppose that there exists a faithful
$Z$-module $M$ with the property that for all non-zero $m\in Z$, $M/mM$ is
finite. Then $Z$ is a domain.
\end{lemma}
\begin{proof}
Let $a,b\in Z$ be non-zero. We have an exact sequence of $Z$-modules
$$
M/bM \stackrel{a}{\rar} M/abM \rar M/aM\rar 0.
$$
The left and right terms are finite by assumption, so $M/abM$ is finite.
But since $Z$ is infinite, and $M$ is a faithful module, $M$ is also infinite,
and so $ab\neq 0$.
\end{proof}\noindent
The following result gives a description of the rings $Z$ that
occur in Theorem \ref{thm:mainintro}.
\begin{theorem}\label{thm:R}
Let $\cZ$ be an infinite commutative ring. Then the following
assertions are equivalent:
\begin{enumerate}[leftmargin=*, label={\upshape(\alph*)}]
\item for each non-zero $m\in \cZ$, the ring $\cZ/m\cZ$ is finite;
\item the ring $\cZ$ is a domain, and each non-zero prime
ideal $\fp$ of $\cZ$ is finitely generated as an ideal and has
finite index in~$\cZ$;
\item either $\cZ$ is a field, or it is a one-dimensional
noetherian domain with the property that for every maximal ideal
$\fm$ of $\cZ$ the field $\cZ/\fm$ is finite.
\end{enumerate}
\end{theorem}
\begin{proof}
First we prove that (a) implies (b).
From (a) it follows, by Lemma \ref{lem:finquodomain} applied to $M=Z$,
that $Z$ is a domain. Now let $\fp$ be a non-zero
prime ideal, and let $m\in\fp$ be non-zero. Then we
have $mZ\subset\fp\subset Z$, and since $Z/mZ$ is
finite, the index of $\fp$ in $Z$ is finite and
$\fp/mZ$ is finite. Hence $\fp$ is generated by
$m$ together with a finite set, and is therefore
finitely generated.

Now we prove that (b) implies (c). By \cite[Theorem 2]{Cohen},
each commutative ring of which every prime ideal is
finitely generated is noetherian. Hence (b) implies
that $Z$ is noetherian. If $\fp$ is a non-zero prime
ideal, then $Z/\fp$ is a finite domain, and therefore
a field. Hence each non-zero prime ideal is maximal,
so $Z$ has Krull dimension $0$ or $1$; in the former
case it must be a field.

Finally, suppose that (c) holds. Then, we will deduce (a) by showing that
for any non-zero ideal $I$ of $Z$, the ring $Z/I$ is finite. Suppose that
there exists a non-zero ideal $I$ in $Z$ such that $Z/I$ is infinite. Since
$Z$ is noetherian, we may, without loss of generality, assume that $I$ is
maximal among ideals with this property. So $Z/I$ is infinite, but its quotient
by any non-zero ideal is finite. It follows from Lemma \ref{lem:finquodomain},
applied to $M=Z/I$, that $Z/I$ is a domain, so $I$ is a prime ideal of $Z$.
It is also non-zero, so (c) implies that $I$ is maximal, and therefore that $Z/I$ is
finite, which is a contradiction.
\end{proof}

\section{On the units of semisimple rings}\label{sec:unitssemisimple}\noindent
By a \emph{division ring} we mean a ring $D$ with the property
$D^\times=D\backslash\{0\}$. If $D$ is a division ring and $n$ is a
positive integer, then $\M(n,D)$ denotes the ring of $n$ by $n$ matrices over $D$.
If $G$ is a group, then $G_{\ab}$ denotes
the maximal abelian quotient of $G$.

\begin{lemma}\label{lem:S}
Let $n$ be a positive integer, let $D$ be a division ring, and
for $x\in D^\times$ and $j\in\{1,\dots,n\}$, let $\delta_j(x)\in \M(n,D)$ be the
diagonal matrix with $j$th entry equal to $x$ and all other
entries equal to~$1$. Then each map $\delta_j$ is a group
homomorphism $D^\times\to\M(n,D)^\times$, they all induce the same group
homomorphism $\bar\delta\colon D^\times_{\ab}\to\M(n,D)^\times_{\ab}$, and if
$n\ne2$ or $\#D\ne2$ then $\bar\delta$ is surjective.
\end{lemma}
\begin{proof}
It is clear that each $\delta_j$ is a group
homomorphism, and that for each $x$ all $\delta_j(x)$ are
conjugate to each other, so all $\delta_j$ induce the same
map $D^\times_{\ab}\to\M(n,D)^\times_{\ab}$.
It is evidently surjective if $n=1$.

For $i$, $j\in\{1,\dots,n\}$, $i\ne j$, and $x\in D$, let
$B_{ij}(x)\in\M(n,D)$ be the matrix obtained from the unit
matrix by replacing the $(i,j)$-entry by~$x$; then one has
$B_{ij}(x)\in\M(n,D)^\times$. The subgroup of $\M(n,D)^\times$
generated by all $B_{ij}(x)$ is denoted by $\SL_n(D)$.

By \cite[Chapter IV, Theorem 4.1]{Artin}, we have
$\M(n,D)^\times=\SL_n(D){\cdot}\delta_n(D^\times)$. Each
$B_{ij}(x)$ is a transvection of the right $D$-vector space
$D^n$ in the sense of \cite[Chapter IV, Definition 4.1]{Artin}.
Assume now that $n>2$ or $\#D\ne2$. Then by
\cite[Chapter IV, Section 2]{Artin} each transvection belongs to
$[\M(n,D)^\times,\M(n,D)^\times]$, so
$\SL_n(D)\subset[\M(n,D)^\times,\M(n,D)^\times]$, and therefore
$$
\M(n,D)^\times=[\M(n,D)^\times,\M(n,D)^\times]{\cdot}\delta_n(D^\times).
$$
This implies that $\bar\delta$ is surjective.
\end{proof}

\begin{lemma}\label{lem:T}
Let $n$ be a positive integer, let $D$ be a division ring, and
for each $x\in D^\times$ let $\iota(x)\in\M(n,D)^\times$ be $x$ times
the identity matrix. Then $\iota$ is a group homomorphism
$D^\times\to\M(n,D)^\times$, and the group
$$
\M(n,D)^\times/(\iota(D^\times){\cdot}[\M(n,D)^\times,\M(n,D)^\times])
$$
is abelian of exponent dividing~$n$.
\end{lemma}
\begin{proof}
It is clear that $\iota$ is a group homomorphism.
If we have $n=\#D=2$, then $\M(n,D)^\times$ is a non-abelian group
of order $6$, in which case $\M(n,D)^\times_{\ab}$ has order $2$ and the
conclusion of the lemma is valid. Assume now that $n\ne2$ or $\#D\ne2$,
so that the map $\bar\delta$ from Lemma \ref{lem:S} is surjective.

Denote by $\bar\iota\colon D^\times_{\ab}\to\M(n,D)^\times_{\ab}$ the map
induced by $\iota$. For each $x\in D^\times$ one has
$\iota(x)=\prod_{j=1}^n\delta_j(x)$ and therefore
$\bar\iota(x)=\bar\delta(x)^n$, so the surjectivity of
$\bar\delta$ yields
$$\bar\iota(D^\times_{\ab})=\bar\delta(D^\times_{\ab})^n =(\M(n,D)^\times_{\ab})^n,$$
and the lemma is proved.
\end{proof}\noindent
The following result is due to Wedderburn \cite{Wedderburn}.
\begin{theorem}\label{thm:U}
Let $D$ be a division ring with centre $\Zz(D)$, let
$a\in D$, and let $f\in\Zz(D)[X]$ be an irreducible
polynomial with leading coefficient $1$ such that $f(a)=0$.
Put $l=\deg f$.
Then there exist $b_1$, $b_2$, \dots, $b_l\in D^\times$
such that in $D[X]$ one has
$$f=(X-b_1ab_1^{-1})\cdot\ldots\cdot(X-b_lab_l^{-1}).$$
\end{theorem}
\begin{proof}
See \cite[Theorem (16.9)]{Lam}.
\end{proof}
\begin{lemma}\label{lem:V}
Let $D$ be a division ring that has finite vector space
dimension $m^2$ over its centre $\Zz(D)$, where $m$ is a
positive integer. Then the group
$D^\times/(\Zz(D)^\times{\cdot}[D^\times,D^\times])$ is abelian of exponent
dividing~$m$.
\end{lemma}
\begin{proof}
Since $D^\times/(\Zz(D)^\times{\cdot}[D^\times,D^\times])$ is a
quotient of $D^\times_{\ab}$, it is an abelian group. Let $a\in D^\times$.
It will suffice to show that the image $\bar a$ of $a$ in the quotient
$D^\times/(\Zz(D)^\times{\cdot}[D^\times,D^\times])$ has order dividing~$m$.
The subfield $\Zz(D)(a)$ of $D$ is contained in a maximal
subfield of~$D$, and each maximal subfield of $D$ is an
extension field of $\Zz(D)$ of degree~$m$, by~\cite[(7.22)]{CR1}. Hence
we have $[\Zz(D)(a):\Zz(D)]=l$ for some divisor $l$ of~$m$,
and $a$ is a zero of an irreducible polynomial $f\in\Zz(D)[X]$
of degree~$l$ with leading coefficient~$1$. Using Theorem \ref{thm:U}
we find $b_1$, \dots, $b_l\in D^\times$ such that
$b_1ab_1^{-1}\cdot\ldots\cdot b_lab_l^{-1}=(-1)^lf(0)\in\Zz(D)^\times$.
Mapping this identity to the abelian group
$D^\times/(\Zz(D)^\times{\cdot}[D^\times,D^\times])$ we obtain $\bar a^l=1$, so
$\bar a^m=1$, as required. This proves Lemma \ref{lem:V}.
\end{proof}\noindent
We can now prove Theorem \ref{thm:expintro} and deduce Theorem
\ref{thm:finexpintro}. We recall the statements.
\begin{theorem}\label{thm:W}
Let $k$ be a field, and let $B$ be a central simple algebra over~$k$.
Let the dimension of $B$ as a vector space over $k$ be $d^2$, where
$d$ is a positive integer. Then the group $B^\times/(k^\times[B^\times,B^\times])$ is
abelian of exponent dividing~$d$.
\end{theorem}
\begin{proof}
By \cite[\S 14, Theorem 1]{Alg8}, there are a positive integer $n$ and
a division ring $D$ with $\Zz(D)=k$ such that $B$ is, as an algebra
over $k$, isomorphic to $\M(n,D)$. Then $D$ has finite degree $m^2$
over $k$, and $nm=d$. By Lemma \ref{lem:V}, the cokernel of the natural group
homomorphism $k^\times\to D^\times_{\ab}$ has exponent dividing $m$, and by
Lemma \ref{lem:T} the cokernel of the natural group homomorphism
$D^\times_{\ab}\to\M(n,D)^\times$
has exponent dividing~$n$. It follows that the cokernel of the natural
group homomorphism $k^\times\to\M(n,D)^\times_{\ab}$ has exponent dividing
$nm=d$.
\end{proof}

\begin{theorem}\label{thm:X}
Let $B$ be a semisimple ring that is finitely generated as a module
over its centre $\Zz(B)$. Then
$B^\times/(\Zz(B)^\times[B^\times,B^\times])$ is an abelian group of
finite exponent.
\end{theorem}
\begin{proof}
In the case the semisimple ring $B$ is simple, our
hypothesis that it be finite over its centre implies that it is a
central simple algebra over $\Zz(B)$, and the assertion follows from Theorem
\ref{thm:W}.
Generally, by \cite[Chapter 1, Theorem (3.5)]{Lam} the ring $B$ is a product
of finitely many semisimple rings that are simple, and the result follows from
the case we just did.
\end{proof}


\section{Skew correspondences as morphisms}\label{sec:skewcorres}\noindent
As announced in the introduction, in this section we elaborate upon an argument
of Serre (see e.g. \cite[Tag 0B0J]{stacks-project}) to prove an equivalence between two categories of modules.
The main result of the section is Theorem \ref{thm:equivcat}.
We will need the notion of a skew correspondence (Definition \ref{def:skew}),
and the constructions of the categories $\cC_{\skew}$ and $\cC_{\com}$
(Definition \ref{def:skewcom}).

\begin{notation}\label{not:Z}
The following assumptions will be in force throughout the present section:
$Z$ is an infinite commutative ring that satisfies the equivalent conditions of
Theorem \ref{thm:R}, with field of fractions $Q$; further, $A$ is a
$Q$-algebra of finite vector space dimension over $Q$, and $R$ is a
left-noetherian sub-$Z$-algebra of $A$ with the property that $Q\cdot R = A$.
By an $R$-module we shall always mean a left $R$-module.
We call a module
\emph{finite} if its cardinality is finite.
If $L$ is a finitely generated $R$-module, let $L_{\tors}$ denote
the set of all elements of $L$ that have a non-zero annihilator
in $Z$. Since the image of $Z$ in $R$ is central, $L_{\tors}$ is an
sub-$R$-module.
\end{notation}\noindent
We remark that the hypotheses of Section \ref{sec:2} on the category $\cC$
are satisfied for the category of finitely generated $R$-modules.
We will tacitly use this fact throughout the rest of the paper.

\begin{lemma}\label{lem:fintors}
Let $L$ be a finitely generated $R$-module, and let $U$ be a sub-$R$-module.
Then $U$ is finite if and only if it is contained in $L_{\tors}$.
\end{lemma}
\begin{proof}
First, we show that $L_{\tors}$ is finite.
Since $R$ is left-noetherian, $L_{\tors}$ is finitely generated as an $R$-module.
So there exists a non-zero $m\in \cZ$ that annihilates $L_{\tors}$, and
$L_{\tors}$ is then a finitely generated module over the ring $R/mR$,
which is finite by Lemma \ref{lem:P}. This proves one implication.

For the converse, let $U\subset L$ be a finite
sub-$R$-module. Then for each $x\in U$, the set
$\{zx:z\in Z\}$ is finite, so the annihilator of $x$
in $Z$ has finite index in $Z$; in particular it is
non-zero, since $Z$ is assumed to be infinite, so
$x\in L_{\tors}$.
\end{proof}

\begin{theorem}\label{thm:parta}
Let $L$, $M$ be two finitely generated $R$-modules. Then there exists an
isogeny of $R$-modules $L\rar M$ if and only if there exists a
commensurability of $R$-modules
$L\rightleftharpoons M$, and if and only if there exists an
isomorphism of $A$-modules $Q\otimes_{\cZ} L\cong Q\otimes_Z M$.
\end{theorem}
\begin{proof}
First, suppose that $f\colon L\rar M$ is an isogeny. Then $c_f=(L,\id,f)\colon
L\rightleftharpoons M$ is a commensurability.

Next, suppose that we have a commensurability 
$(X,f,g)\colon L\rightleftharpoons M$. Then the kernels and cokernels of $f$, $g$ are
finite $R$-modules, and so are $Z$-torsion modules by Lemma \ref{lem:fintors}.
They are therefore annihilated by the functor
$Q \otimes_{\cZ} -$, so the maps $\cQ\otimes_{\cZ}f$ and
$\cQ\otimes_{\cZ}g$ are isomorphisms.

Finally, suppose that we have an isomorphism
$\phi\colon \cQ\otimes_{\cZ}L\rightarrow \cQ\otimes_{\cZ}M$
of $A$-modules. It follows from Lemma \ref{lem:fintors} that the quotient
map $L\rar \bar{L}=L/L_{\tors}$ is an isogeny. Since $\bar{L}$ is
$\cZ$-torsion free, it embeds into $\cQ\otimes_{\cZ}L$, and similarly
for $\bar{M}$. By ``clearing denominators'', we can find non-zero elements
$m_1$, $m_2\in \cZ$ such that
$m_1\phi(\bar{L})$ is contained in $\bar{M}\subset \cQ\otimes_{\cZ} M$, and
$\phi(\bar{L})$ contains $m_2\bar{M}$. Since $\bar{M}/m_1m_2\bar{M}$ is finite
by Lemma \ref{lem:P}, it follows that $m_1\phi\colon \bar{L}\rar \bar{M}$
is an isogeny. Let $m_3\in \cZ$ be a non-zero element that annihilates $M_{\tors}$.
Then $m_3M$ is canonically isomorphic to $\bar{M}$, and since $M/m_3M$ is
finitely generated and torsion, Lemma \ref{lem:fintors} implies that the
embedding $\bar{M}\cong m_3M\subset M$ is an isogeny. The composition 
of the three isogenies $L\rar \bar{L}\rar \bar{M}\rar M$ is an isogeny by
Proposition \ref{prop:isog}, as claimed.
\end{proof}
\begin{lemma}\label{lem:welldefined}
Let $L$, $M$ be finitely generated $R$-modules, and let $(X,f,g)$ and $(Y,h,j)\colon
L\rightleftharpoons M$ be equivalent skew correspondences. Let $Q\otimes_{Z} f$
denote the map of $A$-modules $Q\otimes_Z L\rar Q\otimes_Z M$ induced by $f$,
and similarly for $g,h,j$. Then $(\cQ\otimes g)\circ (\cQ\otimes f)^{-1}=
(\cQ\otimes j)\circ (\cQ\otimes h)^{-1}$.
\end{lemma}
\begin{proof}
Let $(W,p,q)\colon X\rightleftharpoons Y$ be an equivalence between $(X,f,g)$
and $(Y,h,j)$. Since $p$ and $q$ are isogenies, Lemma \ref{lem:fintors}
implies that $Q\otimes_Zp$ and $Q\otimes_Zq$ are both
invertible. Moreover, we have
\beq
Q\otimes_Z f & = & (Q\otimes_Z h)\circ (Q\otimes_Z q)\circ (Q\otimes_Z p)^{-1},\\
Q\otimes_Z g & = & (Q\otimes_Z j)\circ (Q\otimes_Z q)\circ (Q\otimes_Z p)^{-1},\\
\eeq
so $(\cQ\otimes g)\circ (\cQ\otimes f)^{-1}=(\cQ\otimes j)\circ (\cQ\otimes h)^{-1}$.
\end{proof}\noindent
Let $\RMod$, respectively $\AMod$ denote the category of finitely generated
$R$-modules, respectively finitely generated $A$-modules.
By Lemma \ref{lem:welldefined}, we may define a functor $\cF$ from $\RMod_{\skew}$
to $\AMod$ by sending an $R$-module $L$ to the $A$-module $Q\otimes_Z L$, and
an equivalence class of skew correspondences represented by
$(X,f,g)\colon L\rightleftharpoons M$ to the map of $A$-modules
$(\cQ\otimes g)\circ (\cQ\otimes f)^{-1}\colon Q\otimes_Z L\rar Q\otimes_Z M$.
The verification that $\cF$ respects composition of morphisms, and thus
does define a functor, is easy and is left to the reader.
\begin{theorem}\label{thm:equivcat}
The functor $\cF\colon \RMod_{\skew}\rar \AMod$ is an equivalence of categories.
\end{theorem}\noindent
To prove the theorem, we will show in the next three lemmas that the functor
$\cF$ has dense image, is full, and is faithful.

\begin{lemma}\label{lem:dense}
Any element of $\AMod$ is isomorphic to $\cF(L)$ for some $R$-module $L$.
\end{lemma}
\begin{proof}
Let $V$ be an $A$-module with finite
generating set $S$. Let $L$ be the sub-$R$-module of $V$ generated
by $S$ over $R$. Then the $A$-module $\cF(L)$ is isomorphic to $V$.
\end{proof}

\begin{lemma}\label{lem:full}
Let $L$, $M$ be finitely generated $R$-modules, and let $\phi\colon \cF(L)\rar \cF(M)$
be a morphism of $A$-modules. Then there exists a skew correspondence
$c\colon L\rightleftharpoons M$ such that $\cF(c)=\phi$.
\end{lemma}
\begin{proof}
Let $\bar{L}$ be the image of $L$ in $Q\otimes_Z L$, and let $\bar{M}$
be the image of $M$ in $Q\otimes_Z M$. By Lemma \ref{lem:fintors},
the natural map $f\colon L\rar Q\otimes_Z L$ gives rise to a commensurability
$c_L=(L,\id,f)\colon L\rightleftharpoons \bar{L}$, and similarly we have a commensurability
$c_M\colon M\rightleftharpoons \bar{M}$. Since $\bar{L}$ and $\bar{M}$
are finitely generated as $R$-modules, and since $\bar{M}$ generates
$Q\otimes_Z M$ over $Q$, we may choose a non-zero
$m\in Z$ such that $m\phi(\bar{L})$ is contained in $\bar{M}$. Let $g$ be the inclusion
$m\phi(\bar{L})\subset \bar{M}$, and define the correspondence
$c_{\phi}=(\bar{L},m,gm\phi)\colon \bar{L}\rightleftharpoons\bar{M}$. It follows
from Lemma \ref{lem:P} that $c_{\phi}$ is a skew correspondence.
By Proposition \ref{prop:composition}, the composition
$c=c_M^{-1}\circ c_{\phi}\circ c_L\colon L\rightleftharpoons M$ is also a skew
correspondence, and it is easy to see that $\cF(c)=\phi$.
\end{proof}

\begin{lemma}\label{lem:faithful}
Let $L$, $M$ be finitely generated $R$-modules, and let $c$, $d\colon L\rightleftharpoons M$
be two skew correspondences such that $\cF(c)=\cF(d)$. Then $c$ and $d$ are
equivalent.
\end{lemma}
\begin{proof}
Let $c=(X,f,g)$, and $d=(Y,h,j)$. We will show that $c$ and $d$ are equivalent
by showing that the fibre product $(X\times_{L\oplus M}Y,p_0,p_1)\colon X\rightleftharpoons Y$
is a commensurability.

First, assume that the images of $f$, $g$, $h$, and $j$ are $Z$-torsion free.
Then $f$ and $g$ factor through $X/X_{\tors}$, and similarly for $h$ and $j$. By Lemma
\ref{lem:fintors}, the quotient maps $X\rar X/X_{\tors}$ and
$Y\rar Y/Y_{\tors}$ are isogenies, so after replacing $c$ and
$d$ by equivalent commensurabilities, we may
assume that $X$ and $Y$ are $Z$-torsion free. It then follows from
Lemma \ref{lem:fintors} that
$f$, $g$, $h$, and $j$ are injective. Since $\cF(c)=\cF(d)$, we have
$$
(\cQ\otimes_{\cZ}g)\circ (\cQ\otimes_{\cZ}f)^{-1}
=  (\cQ\otimes_{\cZ}j)\circ (\cQ\otimes_{\cZ}h)^{-1},
$$
and it follows that the canonical injection
$X\times_{L\oplus M}Y\rar X\times_L Y$ is an isomorphism.
By Proposition \ref{prop:composition}, the fibre product
$(X\times_L Y,p_0,p_1)\colon X\rightleftharpoons Y$ of the diagram
$X\rar M\leftarrow Y$ is a commensurability, which proves this special case
of the lemma.

We now prove the general case. By applying Lemma \ref{lem:fintors}
with $U=f(X)_{\tors}$, and similarly for $g$, $h$, and $j$, we may choose a
non-zero $m\in Z$ such that the images of $mf$, $mg$, $mh$, and $mj$ are
$Z$-torsion free. It is easy to see that $c$ is equivalent to
$(X,mf,mg)$, and $d$ is equivalent to $(Y,mh,mj)$.
So the general case follows from the special case above.
\end{proof}

\begin{proof}[Proof of Theorem \ref{thm:equivcat}]
The result follows by combining Lemmas
\ref{lem:dense}, \ref{lem:full}, and \ref{lem:faithful}.
\end{proof}\noindent
Recall from Theorem \ref{thm:G_L} that if $L$ is a finitely generated
$R$-module, we let $\cG_L$ denote the group
of equivalence classes of commensurabilities $L\rightleftharpoons L$ under
composition. It may be viewed as the full subgroupoid of $\RMod_{\com}$
whose only object is $L$.

\begin{corollary}\label{cor:G_L}
Let $L$ be a finitely generated $R$-module. Then the map
$\cG_L \rar \Aut_A(Q\otimes_Z L)$, $(X,f,g)\mapsto (Q\otimes g)\circ (Q\otimes f)^{-1}$
is a group isomorphism.
\end{corollary}
\begin{proof}
By Proposition \ref{prop:maxgroupoid}, the category $\RMod_{\com}$ is the maximal
subgroupoid of $\RMod_{\skew}$.
So Theorem \ref{thm:equivcat} implies that 
the functor $\cF$ induces an equivalence of categories from $\RMod_{\com}$
to the category whose objects are the finitely
generated $A$-modules, and whose morphisms are the $A$-module isomorphisms.
The corollary follows by restricting $\cF$ to the full subgroupoid $\cG_L$ of
$\RMod_{\skew}$.
\end{proof}

\section{Automorphisms of commensurabilities}\label{sec:autcomm}\noindent
It is in the present section that we construct ring and group
commensurabilities out of module commensurabilities. Here
we retain the assumptions of Notation \ref{not:Z}.

Let $c=(N,f,g)\colon$ $L\rightleftharpoons M$ be a correspondence of $R$-modules.
In the introduction we defined the endomorphism ring of $c$ to be
$\End c=\{(\lambda,\nu,\mu)\in(\End L)\times(\End N)\times(\End M):\lambda f=f\nu$,
$\mu g=g\nu\}$. We also recall the correspondence
$\ee(c)=(\End c,p_0,p_1)\colon \End L\rightleftharpoons \End M$, given by sending
$(\lambda,\nu,\mu)\in \End c$ to $\lambda$ and $\mu$,
respectively, and the induced correspondence of automorphism groups
$\aa(c)\colon \Aut L\rightleftharpoons \Aut M$.
If $f\colon L\rar M$ is an isogeny, we let $c_f$ be the commensurability
$(L,\id,f)\colon L\rightleftharpoons M$, as in Section \ref{sec:2}.

\begin{lemma}\label{lem:endisog}
Let $f\colon L\rar M$ be an
isogeny of finitely generated $R$-modules. Then the correspondence
$\ee(c_f)\colon \End L\rightleftharpoons \End M$ is a commensurability of rings.
\end{lemma}
\begin{proof}
We first show that $p_1$ has finite kernel. We have
$$
\ker p_1 = \{(\lambda,\lambda,0)\in \End L\times\End L\times \End M: f\lambda=0\}
\cong \Hom(L,\ker f),
$$
which is finite since $L$ is finitely generated and $\ker f$ is finite by assumption.

Next, we show that the image of $p_1$ has finite additive index in $\End M$.
The modules $L_{\tors}$ and $M/f(L)$ are finite, so by Lemma \ref{lem:fintors}
there exist non-zero $m_1$, $m_2\in Z$ such that $m_1$ annihilates $L_{\tors}$, and
$m_2$ annihilates $M/f(L)$. Thus, $f\colon m_1L\rar m_1M$ is injective, and the image
contains $m_1m_2M$, so $f^{-1}$ defines a homomorphism
$m_1m_2M\rar m_1L$. Given $\mu\in \End M$, we may therefore define
$\lambda \colon L\rar L, x\mapsto f^{-1}(m_1m_2\mu(f(x)))$, which has the property
that $(\lambda,\lambda,m_1m_2\mu)\in \End c_f$. So the image of $p_1$ contains
$m_1m_2\End M$, which has finite additive index in $\End M$ by Lemma \ref{lem:P}.
This proves that $p_1$ is an isogeny.

We now show that the image of $p_0$ has finite additive index in $\End L$.
Given any $\lambda\in \End L$, we may define $\mu\colon M\rar M$,
$y\mapsto f(\lambda(f^{-1}(m_1m_2y)))$, where $m_1,m_2$ are as before.
We then have
$(m_1m_2\lambda,m_1m_2\lambda,\mu)\in\End c_f$. So the image of $p_0$ contains $m_1m_2\End L$,
which has finite additive index in $\End L$ by Lemma \ref{lem:P}.

Finally, we show that $p_0$ has finite kernel.
We have
\begin{align*}
\ker p_0 & =\{(0,0,\mu)\in \End L\times\End L\times \End M: \mu f=0\}\\
& \cong \Hom(M/f(L),M)\cong
\Hom(M/f(L),M_{\tors}),
\end{align*}
where the last isomorphism follows from Lemma \ref{lem:fintors} and the assumption
that $M/f(L)$ is finite. Invoking Lemma \ref{lem:fintors} again,
it follows that $\ker p_0$ is finite, so $p_0$ is an isogeny.
\end{proof}

\begin{theorem}\label{thm:partb}
Let $L$, $M$ be finitely
generated $R$-modules. Then
for any commensurability $c=(X,f,g)\colon L\rightleftharpoons M$,
the correspondence $\ee(c)\colon \End L\rightleftharpoons \End M$
is a ring commensurability, and the induced correspondence
$\aa(c)\colon$ $\Aut L\rightleftharpoons \Aut M$
is a group commensurability.
\end{theorem}
\begin{proof}
The correspondence $\ee(c)$ is canonically isomorphic to
the composition of $\ee(c_f)^{-1}\colon \End L\rightleftharpoons\End X$
with $\ee(c_g)\colon \End X\rightleftharpoons\End M$. The
correspondences $\ee(c_f)$ and $\ee(c_g)$ are commensurabilities
by Lemma \ref{lem:endisog}, so $\ee(c)$ is a commensurability by
Proposition \ref{prop:composition}.
The assertion on $\aa(c)$ follows from Theorem \ref{thm:O} by passing to
the unit groups.
\end{proof}

\begin{theorem}\label{thm:compos}
Let $c\colon L\rightleftharpoons M$,
$d\colon M\rightleftharpoons N$ be commensurabilities of $R$-modules. Then:
\begin{enumerate}[leftmargin=*, label={\upshape(\alph*)}]
\item the ring commensurability $\ee(d\circ c)\colon \End L \rightleftharpoons \End N$
is equivalent (see Definition \ref{def:equiv}) to the composition of ring
commensurabilities $\ee(d)\circ \ee(c)$,
and the group commensurability $\aa(d\circ c)$ is equivalent to the composition
$\aa(d)\circ \aa(c)$;
\item we have
\begin{eqnarray*}
\ii(\ee(d\circ c))&=&\ii(\ee(d)) \ii(\ee(c)),\\
\ii(\aa(d\circ c))&=&\ii(\aa(d)) \ii(\aa(c)).
\end{eqnarray*}
\end{enumerate}
\end{theorem}
\begin{proof}

(a)
Write $c=(X,f,g)\colon L\rightleftharpoons M$, $d=(Y,h,j)\colon$ $M\rightleftharpoons N$.
We claim that there is an isogeny
$$
i\colon\End c\times_{\End M}\End d \rar \End (d\circ c)
$$
that makes the following
diagram of endomorphism rings commute:

$$
\xymatrix@R=12ex{
& \End c\times_{\End M}\End d\ar@{.>}[d]^{i} \ar[dl]\ar[dr] &\\
\End c \ar[d]\ar[dr] & \End (d\circ c) \ar[dl]\ar[dr]  & \End d \ar[dl]\ar[d] & \\
\End L & \End M & \End N,
}
$$
where all unlabelled morphisms are the ones defined in the introduction.

An element of $\End c\times_{\End M}\End d$ is a pair
of triples
\beq
((\lambda,\xi,\mu),(\mu',\upsilon,\nu)),\\
\lambda\in\End L,\xi\in\End X,\mu,\mu'\in \End M,
\upsilon\in\End Y,\nu\in\End N,
\eeq
satisfying $\lambda f=f\xi$, $\mu g=g\xi$, $\mu' h = h\upsilon$, $\nu j = j\upsilon$,
and the fibre product condition in fact demands that $\mu=\mu'$.

An element of $\End(d\circ c)$ is a triple $(\lambda',\zeta',\nu')\in\End L\times\End (X\times_M Y)
\times \End N$ satisfying $\lambda'fp_0 = \zeta'fp_0$, $\nu'jp_1=jp_1\zeta'$,
where $p_0$, $p_1$ are the canonical projection maps from
$X\times_M Y$ to $X$, respectively $Y$.
Define
\beq
i\colon & \End c\times_{\End M}\End d & \rar & \End(d\circ c)\\
& ((\lambda,\xi,\mu),(\mu,\upsilon,\nu)) & \mapsto & (\lambda,(\xi,\upsilon),\nu).
\eeq
A routine verification, which we leave to the reader,
shows that the image of $i$ is indeed contained in $\End(d\circ c)$.

To see that this definition of $i$ makes the above diagram of endomorphism
rings commute is also routine, and will also be omitted. It remains to
check that $i$ is an isogeny.
The correspondence $\ee(d)\circ \ee(c)\colon \End L\rightleftharpoons
\End N$ consists of $\End c\times_{\End M}\End d$, together with the
maps to $\End L$ and $\End N$. By
Theorem \ref{thm:partb}, the correspondences $\ee(c)$ and $\ee(d)$ are commensurabilities,
so by Proposition \ref{prop:composition} the correspondence
$\ee(d)\circ \ee(c)$ is a commensurability.
In particular, the morphism $\End c\times_{\End M}\End d\rar\End L$ is an
isogeny.
Also, $\End(d\circ c)\rar\End L$ is an isogeny by Theorem \ref{thm:partb}. The
fact that $i$ is an isogeny therefore follows from Proposition \ref{prop:isog}.
This proves our claim.

The isogeny $i$ defines an equivalence between $\ee(d\circ c)$ and
$\ee(d)\circ \ee(c)$. This proves part (a) for endomorphism rings.
By passing to the unit groups and applying Theorem \ref{thm:O} to the
isogeny $i$, we also obtain part (a) for automorphism groups.
  
Part (b) immediately follows from part (a) by Propositions \ref{prop:equiv}
and \ref{prop:composition}.
\end{proof}
\begin{proposition}\label{prop:equivequiv}
Let $L$, $M$ be finitely generated $R$-modules, and let $c$, $d\colon$
$L\rightleftharpoons M$ be two
commensurabilities. If $c$ is equivalent to $d$, then $\ee(c)$ is equivalent
to $\ee(d)$, and $\aa(c)$ is equivalent to $\aa(d)$.
\end{proposition}
\begin{proof}
Let $c=(X,f,g)$ and $d=(Y,h,j)$. First, assume that an equivalence between
$c$ and $d$ is given by an isogeny $p\colon Y\rar X$, so that we have the
following commutative diagram:
$$
\xymatrix{
& Y\ar[d]_{p}\ar[dl]_{h}\ar[dr]^j & \\
L & X\ar[r]^{g}\ar[l]_{f} & M. \\
}
$$
As before, write $c_f=(X,\id,f)\colon X\rightleftharpoons L$,
and define $c_g$, $c_p$ similarly. Then
$d$ is canonically isomorphic to $(c_g\circ c_p)\circ (c_p^{-1}\circ c_f^{-1})$.
By Theorem \ref{thm:compos}, the commensurability $\ee(d)$ is equivalent to
$\ee(c_g)\circ \ee(c_p)\circ \ee(c_p)^{-1}\circ \ee(c_f^{-1})$.
By Proposition \ref{prop:inverse}, the composition $\ee(c_p)\circ \ee(c_p)^{-1}$
is equivalent to $(\End X,\id,\id)\colon \End X\rightleftharpoons \End X$.
So by Proposition \ref{prop:equivcompos} the commensurability
$\ee(c_g)\circ \ee(c_p)\circ \ee(c_p)^{-1}\circ \ee(c_f^{-1})$
is equivalent to $\ee(c_g)\circ (\End X,\id,\id)
\circ \ee(c_f^{-1})$, which is canonically isomorphic to
$\ee(c_g) \circ \ee(c_f^{-1})$. Applying Theorem \ref{thm:compos} again,
we find that $\ee(c_g) \circ \ee(c_f^{-1})$ is equivalent
to $\ee(c_g\circ c_f^{-1})$. Finally, $c_g\circ c_f^{-1}$ is canonically
isomorphic to $c$, and the special case of the proposition follows.

Passing to the general case, let $(W,p,q)\colon X\rightleftharpoons Y$ 
be an equivalence between $c$ and $d$.
Since $p$ is an isogeny, $c$ is equivalent to $(W,fp,gp)$, and
since $q$ is an isogeny, $d$ is equivalent to $(W,hq,jq)=(W,fp,gp)$.
The result therefore follows from the special case we just did.
\end{proof}\noindent
Let $\Rng$ denote the category of rings, and $\Grp$ the category of groups.
Theorem \ref{thm:compos} and Proposition \ref{prop:equivequiv} imply that
there is a functor from $\RMod_{\com}$ to $\Rng_{\com}$ that
takes an $R$-module $L$ to the ring $\End L$, and an equivalence class of
$R$-module commensurabilities,
represented by a commensurability $c$, to the equivalence class of
ring commensurabilities represented by $\ee(c)$. Further, Theorem \ref{thm:O}
shows that we have
the functors $^+$ and $^\times$ from $\Rng_{\com}$ to
$\Grp_{\com}$ which take a ring to the additive, respectively multiplicative
group of the ring. Finally, Propositions \ref{prop:equiv} and \ref{prop:composition}
imply that we have the functor $\ii$ from $\Grp_{\com}$ to the group $\Q_{>0}$, thought of
as a groupoid with one object. To summarise, we have the functors of groupoids
\begin{eqnarray}\label{eq:groupoids}
\RMod_{\com} \stackrel{\End}{\longrightarrow} \Rng_{\com}\doublerightarrow{+}{\times} \Grp_{\com}
\stackrel{\ii}{\longrightarrow} \Q_{>0}.
\end{eqnarray}

Let $L$ be a finitely generated $R$-module, and let $V$ denote the $A$-module
$Q\otimes_Z L$. The isomorphism of Corollary \ref{cor:G_L} and the functors
(\ref{eq:groupoids}) then induce group homomorphisms
\beqn\label{eq:homs}
\Aut_A V\cong & \cG_L & \rar & \Q_{>0},\\
& c &\mapsto & \ii(\ee(c))\quad\text{ and}\\
& c & \mapsto & \ii(\aa(c)).
\eeqn

\begin{lemma}\label{lem:throughtorsion}
Let $L$ be a finitely generated $R$-module, write $\bar{L}=L/L_{\tors}$,
and let $f\colon L\rar \bar{L}$ denote the quotient map.
Then the isomorphism in $\RMod_{\com}$ given by the commensurability
$(L,\id,f)\colon L\rightleftharpoons \bar{L}$
induces an isomorphism $\cG_{L}\rar \cG_{\bar{L}}$ that commutes with
the maps $\cG_L\rar \Q_{>0}$ and $\cG_{\bar{L}}\rar \Q_{>0}$ defined in (\ref{eq:homs}).
\end{lemma}
\begin{proof}
Let $t$ denote the commensurability $L\stackrel{\id}{\leftarrow}L\rar \bar{L}$.
Then the isomorphism $\cG_{L}\rar \cG_{\bar{L}}$
is given by composition on the right with $t$ and on the left with $t^{-1}$.
It follows from Theorem \ref{thm:compos}, Proposition \ref{prop:composition}
and Proposition \ref{prop:equiv} that this isomorphism commutes with the maps
(\ref{eq:homs}).
\end{proof}

\begin{proposition}\label{prop:throughcentre}
Let $L$ be a finitely generated $R$-module, and denote the $A$-module $Q\otimes_Z L$
by $V$. Let $\alpha$ be an element of $\Zz(\End_A V)^\times\subset \Aut_A V\cong \cG_L$.
Then its image in $\Rng_{\com}$ under the first functor of (\ref{eq:groupoids}) is the
identity morphism on $\End L$.
\end{proposition}
\begin{proof}
By Lemma \ref{lem:throughtorsion}, we may assume that $L$ is $Z$-torsion free.
Thus, $L$ injects into $V=Q\otimes_Z L$. For any sub-$R$-module $U$ of $V$, write
$E_U=\{\phi\in \End_A V:\phi U\subset U\}$. Then the injection $L\rightarrowtail V$
induces a map $\End_R L\rar \End_AV$, which is injective and whose image is
exactly $E_L$.

Let $\alpha \in \Aut_AV$ be arbitrary. Then the isomorphism $\Aut_AV\cong \cG_L$
identifies $\alpha$ with the equivalence class of commensurabilities
represented by $c=(L\cap \alpha^{-1}L,i,\alpha)\colon L\rightleftharpoons L$,
where $i\colon L\cap \alpha^{-1}L\rar L$ is the inclusion map.
We have
\begin{align*}
\End c = & \{(\lambda_0,\lambda_1)\in \End_A V\times \End_A V :\\
& \lambda_0 \in E_L\cap E_{\alpha^{-1}L},
\lambda_1 \in E_{\alpha L}\cap E_L,
\lambda_0 = \alpha^{-1}\lambda_1\alpha\}.
\end{align*}
The commensurability $\ee(c)$ is then of the form
$(\End c,p_0,p_1)\colon \End L\rightleftharpoons \End L$,
where $p_0\colon (\lambda_0,\lambda_1)\mapsto \lambda_0$,
and $p_1\colon (\lambda_0,\lambda_1)\mapsto \lambda_1 = \alpha\lambda_0\alpha^{-1}$.

It follows that if $\alpha$ is an element of $\Zz(\End_A V)^\times$,
then $p_0$ and $p_1$ are equal. In this case, the commensurability
$(\End c,\id,p_0)\colon \End c\rightleftharpoons \End L$ defines
an equivalence between $\ee(c)$ and
$(\End L,\id,\id)\colon \End L\rightleftharpoons\End L$, the identity
morphism on $\End L$ in $\Rng_{\com}$.
\end{proof}\noindent
The following result is an immediate consequence of Proposition
\ref{prop:throughcentre}.
\begin{corollary}\label{cor:throughcentre}
The two group homomorphisms $\Aut_AV\rar \Q_{>0}$ of (\ref{eq:homs})
factor through $\Aut_AV/\Zz(\End_A V)^\times$.
\end{corollary}

\begin{remark}\label{rmrk:ie}
The computation in the proof of Proposition \ref{prop:throughcentre} shows that
the group homomorphism $\ii\circ \ee\colon \Aut_AV\rar \Q_{>0}$ is given by
$$
\alpha \mapsto \frac{(E_L:E_{\alpha L}\cap E_L)}{(E_L:E_{L}\cap E_{\alpha^{-1}L})},
$$
and analogously for $\ii\circ \aa$.
\end{remark}

\section{The case of semisimple algebras}\label{sec:proofs}\noindent
In this section, we prove our main results.
We begin with Theorem \ref{thm:mainintro}. We recall the statement.
\begin{theorem}\label{thm:maingeneral}
Let $Z$ be an infinite domain such that for all non-zero $m\in Z$ the ring
$Z/mZ$ is finite, let $Q$ be the field of fractions of $Z$, let
$A$ be a semisimple $Q$-algebra of finite vector space dimension
over $Q$, let $R\subset A$ be a sub-$Z$-algebra with $Q\cdot R=A$,
and let $L$, $M$ be finitely generated $R$-modules.
Then:
\begin{enumerate}[leftmargin=*, label={\upshape(\alph*)}]
\item there is an $R$-module commensurability
$L\rightleftharpoons M$ if and only if the $A$-modules
$Q\otimes_ZL$ and $Q\otimes_ZM$ are isomorphic;
\item if $c\colon L\rightleftharpoons M$ is an $R$-module
commensurability, then $\ee(c)\colon\End L\rightleftharpoons\End M$
is a ring commensurability, and $\aa(c)\colon\Aut
L\rightleftharpoons\Aut M$ is a group commensurability;
\item if $c$, $c'\colon L\rightleftharpoons M$ are
$R$-module commensurabilities, then one has
$$\ii(\ee(c))=\ii(\ee(c')),\;\;\;\;\; \ii(\aa(c))=\ii(\aa(c')).$$
\end{enumerate}
\end{theorem}
\begin{proof}
By Theorem \ref{thm:Q}, the ring $R$ is left-noetherian, so the assumptions
of Notation \ref{not:Z} are satisfied. Parts (a) and (b) of the theorem therefore
follow from Theorems \ref{thm:parta} and \ref{thm:partb}, respectively.

We now prove part (c). Let $c,c'\colon L\rightleftharpoons M$ be $R$-module
commensurabilities. By Theorem \ref{thm:compos}, the assertion of part
(c) is equivalent to the statement that
$$
\ii(\ee(c^{-1}\circ c'))=\ii(\aa(c^{-1}\circ c'))=1.
$$
So we may, without loss of generality, assume that $L=M$, and it
suffices to show that the homomorphisms
$$
\ii\circ \ee, \ii\circ \aa\colon \Aut_A V\cong \cG_L\rar \Q_{>0}
$$
defined
in (\ref{eq:homs}) are trivial. Here $V$ denotes the $A$-module $Q\otimes_Z L$.

Let $B$ denote the $Q$-algebra $\End_AV$, so that $\cG_L=B^\times$.
Since $\Q_{>0}$ is abelian, both homomorphisms $\ii\circ \ee$ and $\ii\circ \aa$
factor through $B^\times/[B^\times,B^\times]$. By Corollary \ref{cor:throughcentre},
they also factor through $B^\times/\Zz(B)^\times$. Since $A$ is a semisimple
ring, and since $V$ is a finitely generated $A$-module, it follows that $V$
is a finite direct sum of
simple modules, so by Schur's lemma $B$ is a direct product of matrix rings over
division rings, and in particular a semisimple ring. By Theorem \ref{thm:X}, the
quotient $B^\times/(\Zz(B)^\times[B^\times,B^\times])$ is an abelian group
of finite exponent. Since $\Q_{>0}$ is torsion-free, any homomorphism
$B^\times/(\Zz(B)^\times[B^\times,B^\times])\rar \Q_{>0}$ must be trivial.
\end{proof}
\begin{example}\label{ex:notss}
The following example demonstrates that if we replace the semisimplicity
assumption on $A$ by the condition that $R$ be left-noetherian, then
the conclusion of Theorem \ref{thm:mainintro}(c) need no longer hold.

Let $R=\begin{psmallmatrix}\Z&\Z\\0&\Z\end{psmallmatrix}$,
and $A=\Q\otimes_{\Z}R$.
Let $L$ be a free $R$-module of rank 1, set $V=\Q\otimes_{\Z}L$, and
$B=\End_AV$. We have $\End L\cong R^{\opp}$, and similarly
\begin{eqnarray*}
B^\times\cong (A^{\opp})^\times \cong \begin{psmallmatrix}\Q^\times&0\\\Q&\Q^\times\end{psmallmatrix}.
\end{eqnarray*}

Recall from equation (\ref{eq:homs}), that $\ii\circ\ee$ defines a group
homomorphism from $B^\times$ to $\Q_{>0}$,
which factors through $B^\times/(\Zz(B)^\times\cdot [B^\times,B^\times])$.
The map $\begin{psmallmatrix}a&0\\b&c\end{psmallmatrix}\mapsto c/a$
defines an isomorphism of this quotient with $\Q^\times$.
For $\alpha=\begin{psmallmatrix}1&0\\0&c\end{psmallmatrix}$,
one easily computes, using Remark \ref{rmrk:ie}, that
$\ii(\ee(\alpha))=\ii(\aa(\alpha))=|c|$.
It follows that both $\ii\circ\ee$ and $\ii\circ\aa$ map
$\begin{psmallmatrix}a&0\\b&c\end{psmallmatrix}$ to $|c/a|$,
and are therefore far from trivial.
\end{example}\noindent
We now deduce Theorem \ref{thm:QGintro}.
\begin{theorem}\label{thm:QG}
Let $G$ be a finite group, let $V$ be a finitely generated
$\Q[G]$-module, and put $\cS=\{L:L$ is a finitely generated
$\Z[G]$-module with $\Q\otimes_\Z L\cong V$ as
$\Q[G]$-modules$\}$. Then there exists a unique function
$\ia\colon\cS\times\cS\to\Q_{>0}$ such that
\begin{enumerate}[leftmargin=*, label={\upshape(\alph*)}]
\item if $L$, $L'$, $M$, $M'\in\cS$ and $L\cong L'$,
$M\cong M'$, then $\ia(L,M)=\ia(L',M')$;
\item if $L$, $M$, $N\in\cS$, then $\ia(L,M)\cdot\ia(M,N)=
\ia(L,N)$;
\item if $M\in\cS$, and $L\subset M$ is a submodule of
finite index, then with $H=\{\sigma\in\Aut M:\sigma L=L\}$ and
$\rho\colon H\to\Aut L$ mapping $\sigma\in H$ to $\sigma|L$, one
has
$$\ia(L,M)={\frac{(\Aut M:H)\cdot\#\ker\rho}{(\Aut L:\rho H)}}.$$
\end{enumerate}
\end{theorem}
\begin{proof}
Existence immediately follows from Theorem \ref{thm:maingeneral}:
for $L$, $M\in \cS$, we may define $\ia(L,M)=\ii(\aa(c))$ for any
commensurability $c\colon L \rightleftharpoons M$.
In particular, property (c) follows by taking the commensurability
$c=(L,\id,i)\colon L\rightleftharpoons M$, where $i\colon L\rightarrow M$ is
the inclusion map, and noting that in this case, $\aa(c)$ is the
commensurability $\Aut L \stackrel{\rho}{\leftarrow}H\rightarrowtail \Aut M$.

To show uniqueness, observe that the
conditions of the theorem imply that the function $\ia$,
if it exists, is uniquely determined by its values on $\Z$-free modules.
Indeed, if $m_1$ and $m_2$ are the exponents of the $\Z$-torsion
submodule of $L$, respectively of $M$, then condition (b) requires that
$$
\ia(m_1L,L)\ia(L,M) = \ia(m_1L,m_2M)\ia(m_2M,M).
$$
Condition (c) determines
the values of $\ia(m_1L,L)$ and $\ia(m_2M,M)$, so $\ia(L,M)$ is determined by
$\ia(m_1L,m_2M)$. Clearly, the modules $m_1L$ and $m_2M$ are both $\Z$-free.

But if $L$, $M$ are $\Z$-free, and $\Q\otimes_{\Z} L\cong_{\Q[G]} \Q\otimes_{\Z} M$,
then there exists an embedding $L\rightarrowtail M$ with finite index,
in which case $\ia(L,M)$ is determined by conditions (a) and (c).
\end{proof}\noindent
The first interesting case of Theorem \ref{thm:QGintro} is already when
$G$ is the trivial group, so that finitely generated $\Z[G]$-modules are just
finitely generated abelian groups.
\begin{proposition}\label{prop:fingengps}
Let $L$, $M$ be finitely generated abelian groups. Then:
\begin{enumerate}[leftmargin=*, label={\upshape(\alph*)}]
\item there exists a commensurability $L\rightleftharpoons M$ if and only if
$L$ and $M$ have the same rank;
\item if $L\cong \Z^n\oplus L_{0}$ and $M\cong \Z^n\oplus M_{0}$,
where $L_{0}$ and $M_{0}$ are finite abelian groups,
then
$$
\ia(L,M) = \frac{(\#M_{0})^n\cdot \#\Aut M_{0}}{(\#L_{0})^n\cdot \#\Aut L_{0}}.
$$
\end{enumerate}
\end{proposition}
\begin{proof}
Part (a) immediately follows from Theorem \ref{thm:mainintro}(a).

We now prove part (b). First we compute $\ia(\Z^n,L)$.
The split exact sequence
$$
0\rar L_{0}\rar L\stackrel{f}{\rar} \Z^n\rar 0
$$
induces a surjective map
$$
\Aut L\rar \Aut L_{0}\times\Aut \Z^n,
$$
whose kernel is easily seen to be canonically isomorphic to $\Hom(\Z^n, L_{0})$.
It follows that if
$c$ is the commensurability $(L,f,\id)\colon\Z^n\rightleftharpoons L$, then
the map $\Aut c\rar \Aut L$ is an isomorphism, while
the map $\Aut c \rar \Aut \Z^n$ is onto, with kernel of cardinality
$\#\Hom(\Z^n, L_{0})\cdot \#\Aut L_{0}=(\#L_{0})^n\cdot \#\Aut L_{0}$.
Hence $\ia(\Z^n,L)=\ii(\aa(c)) = (\#L_{0})^n\cdot \#\Aut L_{0}$.

It follows from the above computation that
\begin{align*}
\ia(L,M)&=\frac{\ia(\Z^n,M)}{\ia(\Z^n,L)}\\
&= \frac{(\#M_{0})^n\cdot \#\Aut M_{0}}{(\#L_{0})^n\cdot \#\Aut L_{0}},
\end{align*}
as claimed.
\end{proof}

\end{document}